\documentclass[12pt]{amsart}
\usepackage{amsfonts}
\usepackage{amssymb}
\usepackage{amsmath}
\usepackage{amsthm}
\usepackage{verbatim}
\usepackage{esint}
\usepackage{color}

\newcommand{\delo}{\partial \Omega}
\renewcommand{\div}{\mathrm{div}}

\newcommand{\dom}{\partial\Omega}

\newcommand{\N}{{\mathbb N}}
\newcommand{\R}{{\mathbb R}}
\newcommand{\Z}{{\mathbb Z}}

\newcommand{\diff}{\ensuremath{\:\text{d}}}

\DeclareMathOperator*{\dist}{dist}
\newcommand{\HS}{\ensuremath{\text{\it HS}^1}}

\newcommand{\divt}{\ensuremath{\text{div}}}

\normalbaselines 

\newtheorem{theorem}{Theorem}[section]
\newtheorem{lemma}[theorem]{Lemma}
\newtheorem{corollary}[theorem]{Corollary}

\newtheorem{definition}[theorem]{Definition}
\numberwithin{equation}{section}

\theoremstyle{definition}

\begin{document}

\title{Boundary value problems for second order elliptic operators satisfying a Carleson condition}

\begin{abstract}
Let $\Omega$ be a Lipschitz domain in $\mathbb R^n$ $n\geq 2,$ and
$L=\divt A\nabla\cdot$ be a second order elliptic operator in
divergence form. We establish solvability of the Dirichlet
regularity problem with boundary data in $H^{1,p}(\dom)$ and of
the Neumann problem with $L^p(\partial\Omega)$ data for the
operator $L$ on Lipschitz domains with small Lipschitz constant.
We allow the coefficients of the operator $L$ to be rough obeying
a certain Carleson condition with small norm. These results
complete the results of \cite{DPP} where the $L^p(\dom)$ Dirichlet
problem was considered under the same assumptions and \cite{DR}
where the regularity and Neumann problems were considered on two
dimensional domains.
\end{abstract}

\author{Martin Dindo\v{s}, Jill Pipher \and David Rule}

\maketitle \markright{BOUNDARY VALUE PROBLEMS FOR OPERATORS
SATISFYING CARLESON CONDITION}

\section{Introduction}

\bigskip

This paper continues the study, begun in \cite{DPP}, of boundary
value problems for second order divergence form elliptic operators,
when the coefficients satisfy a certain natural, minimal smoothness
condition. Specifically, we consider operators $L= \mbox{div}(A
\nabla)$ such that $A(X)=(a_{ij}(X))$ is uniformly elliptic in the
sense that there exists a positive constant $\Lambda$ such that
$$
\Lambda|\xi|^2 < \sum_{i,j} a_{ij}(X)\xi_i\xi_j <
\Lambda^{-1}|\xi|^2,
$$
for all $X$ and all $\vec{\xi}\in \R^n$. 

We do not assume symmetry
of the matrix $A$: the non-symmetric situation requires a different approach from that of the symmetric situation, and moreover, the 
sharp results are very different in this setting as well.  There are a variety of reasons for studying 
non-symmetric operators. These include the connections with
non-divergence form equations, and the broader issue of obtaining
estimates on elliptic measure in the absence of special $L^2$
identities of Rellich type which relate tangential and normal derivatives. Boundary value problems for divergence form
equations under minimal regularity assumptions on the coefficients have been studied for several decades. The study of non-symmetric operators has
been treated fairly recently, in spite of the aforementioned connections and their
relevance to the theory of homogenization (\cite{BLP}).

We do, however, assume some regularity on the coefficients,
in terms of 
$$\mbox{osc}_{B(X,\delta(X)/2)}a_{ij} =   \sup_{X_1,X_2 \in B(X,\delta(X)/2)} |a_{i,j}(X_1) - a_{i,j}(X_2)|,$$
where $\delta(X)$ denotes the distance of $X$ to the boundary.

The main result of this paper is that under the assumption that
\begin{equation}\label{carlI}
d\mu =   \delta(X)^{-1}\left(\mbox{osc}_{B(X,\delta(X)/2)}a_{ij}\right)^2 dX
\end{equation}
is the density of Carleson measure with small Carleson norm (see Definition \ref{cmeasure}), and under certain conditions on $\partial\Omega$, 
the Dirichlet regularity problem for the operator $L$ with boundary data in
$H^{1,p}(\partial\Omega)$ is solvable for the full range $1<p<\infty$. In addition, the Neumann boundary value problem with $L^p(\partial\Omega)$
data is solvable in the same range of $p$.\vglue2mm

We now set the context for these results.
In \cite{KKPT}, the study of nonsymmetric divergence form
operators with bounded measurable coefficients was initiated.
The objective was
to develop new methods to prove mutual absolute continuity of elliptic measure and surface measure, and to apply
these to the study of the Dirichlet problem with data in $L^p$. The non-symmetric operators studied in \cite{KKPT} had coefficients independent of the 
transverse variable, an instance of a problem that has close connections to the Kato square root problem for complex coefficient operators.
Despite the fact that the new methods introduced were applicable in arbitrary dimensions, the application to these divergence form matrices was limited to
two dimensions, and was not completely resolved until much later  (\cite{HKMP}) .

In
\cite{KP}, the methods of \cite{KKPT} were applied to another class of divergence form operators, namely those satisfying the related gradient condition: $d\mu = \delta(X) |\nabla a_{i,j}(X)|^2 dX$
is a Carleson measure.  This regularity condition on the coefficients of the matrix is quite natural. 
In particular, it arises from the pull-back of the Laplacian under a change of variable considered by Dahlberg, Kenig-Stein, and Ne\v{c}as
(\cite{D}, \cite{N}) that 
produces an ``adapted" distance function: a distance function to the boundary of a Lipschitz graph possessing some regularity.
The main result of \cite{KP} is that, for an operator in this class, the elliptic measure and surface Lebesgue measure are mutually absolutely continuous and
in fact there exists solvability of the Dirichlet problem with boundary data in {\it some} $L^p$ space for $p$ sufficiently large.

The sharp range of solvability ($1<p<\infty$) of the $L^p$ Dirichlet boundary value problem was solved in \cite{DPP} for the class of operators under consideration here.
This ``small" Carleson condition on coefficients also arises naturally. 
For example, take
any smooth
elliptic operator in the region above a graph $t = \varphi(x)$.
If the function $\varphi$ is $C^1$, it is a classical result of  \cite{FJR} that the
Dirichlet, regularity and Neumann boundary value problems are
solvable with data in $L^p$ for $1 < p < \infty$, by the method of
layer potentials. 
If the function $\varphi$ satisfies the weaker condition, $\nabla \varphi \in L^{\infty} \bigcap$ {\it
VMO}, then solvability of the Dirichlet problem  in $L^p$ for $1 < p < \infty$ is a corollary of the main theorem of \cite{DPP}.
By changing variables in a solution $u$, via the mapping $\Phi:{\mathbb
R}^n_+\to\{X=(x,t);t>\phi(x)\}$ defined by (\ref{eDKS}), the function
$v=u\circ\Phi$ will solve an
elliptic equation in $\R^n_+$ whose coefficients satisfy
(\ref{carlI}). Thus
our main theorem (together with \cite{DPP}) has the corollary that 
the Dirichlet, regularity, and Neumann problems for smooth operators are solvable, in
the same range of $p$ as in \cite{FJR}, when the boundary of the domain is defined by $t
= \varphi(x)$ where $\nabla \varphi \in L^{\infty} \bigcap$ {\it
VMO}. The exact statements of the results are formulated in Section 2.
\vglue2mm

The regularity and Neumann
boundary value problems for operators satisfying (\ref{carlI}), are considerably more difficult than the Dirichlet problem.
Because the solvability of the Neumann problem is connected to estimates on singular integral operators (as opposed to maximal functions in the Dirichlet problem),  the estimates required when specifying Neumann conditions are more delicate in general. 
For the operators we consider here, progress was first made in two dimensions in
the work of \cite{DR}. The proof relied on a particular notion of
conjugate solution and did not
generalize to higher dimensional domains. (We note, however, that such generalizations have been carried out in other contexts: see Theorem 9.3 of \cite{AA} and Section 1 of \cite{R}.)

Indeed, Theorem \eqref{NPLPneu} is one of very few results for solvability of a Neumann problem in all dimensions for non-symmetric divergence form operators with
minimal regularity.  Presently, there are no known results for solvability of the Neumann problem when the coefficients satisfy a Carleson condition without the smallness assumption on the Carleson norm. One would expect to find solvability for $p$ near $1$ in this case, but the methods of this paper appear to be applicable only when determining solvability for the full range of $p$.

The strategy of our proofs is as follows. We first establish the solvability of the regularity problem for $p=2$.
In \cite{DK} and \cite{DKP} a better understanding of the Dirichlet and regularity boundary
values problems was obtained, including a \lq\lq duality" between
the solvability of the Dirichlet boundary value problem and the
regularity problem for the adjoint operator. In particular, we will be able to infer solvability for all values of $p$ once $p=2$ is established. Hence we are able to avoid the use of the $p$-adapted square
function introduced in \cite{DPP} and which was essential in
establishing solvability of the $L^p$ Dirichlet problem.

The proof of the $p=2$ case of the regularity problem starts with a localization and change of variables: thus, most estimates can be reduced to
their local versions on a neighborhood of $0$ in ${\mathbb R}^{n}_+$. In particular this is how we establish the key estimate (\ref{e1}) of Lemma \ref{l1} (section 3).
Essentially, this estimate and the estimate from Lemma \ref{l2} can be encapsulated as
\begin{eqnarray}
\int_{\partial\Omega}S^2(\nabla u)\,d\sigma \lesssim \int_{\dom}
|\nabla_T u|^2\,d\sigma+\|\mu\|_{Carl}\int_{\dom}N^2(\nabla u)
d\sigma.
\end{eqnarray}
where $S$ and $N$ are the square function and the non-tangential maximal function, respectively. By $\|\mu\|_{Carl}$ we denote the Carleson norm of the coefficients (a variant of (\ref{carlI})). Hence when we show that $S$ and $N$ applied to the gradient of a solution have comparable $L^2$ norms (section 4), we will have, when $\|\mu\|_{Carl}$ is sufficiently small,
\begin{eqnarray}
\int_{\dom}N^2(\nabla u)
d\sigma\approx \int_{\partial\Omega}S^2(\nabla u)\,d\sigma \lesssim \int_{\dom}
|\nabla_T u|^2\,d\sigma\label{R2}
\end{eqnarray}
which is the desired solvability of the regularity problem when $p=2$ (section 5).

The solvability of the Neumann problem is much trickier as there is no
appropriate analogue of the duality results in \cite{DK}. We overcome this
by using the solvability of the regularity problem, which we
established here first, in particular our starting estimate is the $L^p$ version of the estimate (\ref{R2}). Then we proceed by induction for integer values of $p\ge 2$ where, at each step, a rather involved series of integration by parts allows us to introduce one extra co-normal derivative on the right-hand side.  Lemma \ref{NPl2} in section 6 contains this key step. The solvability for non-integer values of $p$ follows by the extrapolation results in \cite{KP3} and \cite{KP2} (section 7).

\subsection*{Acknowledgements} J.~Pipher was partially supported by NSF DMS grant 0901139.
M. Dindo\v{s} was partially supported by EPSRC EP/J017450/1 grant.
D.~Rule gratefully acknowledges the support of CANPDE. Part of
this research was carried out during a visit to Brown University,
M.~Dindo\v{s} and D.~Rule wish to thank the University for its
hospitality.

\section{Definitions and Statements of Main Theorems}

Let us begin by introducing Carleson measures and the square function
on domains which are locally given by the graph of a function. We
shall assume that our domains are Lipschitz.

\begin{definition}
$\Z \subset \R^n$ is an $\ell$-cylinder of diameter $d$ if there
exists a coordinate system $(x,t)$ such that
\[
\Z = \{ (x,t)\; : \; |x|\leq d, \; -2\ell d \leq t \leq 2\ell d \}
\]
and for $s>0$,
\[
s\Z:=\{(x,t)\;:\; |x|<sd, -2\ell d \leq t \leq 2\ell d \}.
\]
\end{definition}

\begin{definition}\label{DefLipDomain}
$\Omega\subset \R^n$ is a Lipschitz domain with Lipschitz
`character' $(\ell,N,C_0)$ if there exists a positive scale $r_0$ and
at most $N$ $\ell$-cylinders $\{{\Z}_j\}_{j=1}^N$ of diameter $d$, with
$\frac{r_0}{C_0}\leq d \leq C_0 r_0$ such that\vglue2mm

\noindent (i) $8{\Z}_j \cap \delo$ is the graph of a Lipschitz
function $\phi_j$, $\|\nabla\phi_j \|_\infty \leq \ell \, ;
\phi_j(0)=0$,\vglue2mm

\noindent (ii) $\displaystyle \delo=\bigcup_j ({\Z}_j \cap \delo
)$,

\noindent (iii) $\displaystyle{\Z}_j \cap \Omega \supset \left\{
(x,t)\in\Omega \; : \; |x|<d, \; \mathrm{dist}\left( (x,t),\delo
\right) \leq \frac{d}{2}\right\}$.
\end{definition}

\begin{definition} Let $\Omega$ be a Lipschitz domain.
For $Q\in\partial\Omega$, $X\in \Omega$ and $r>0$ we write:
\begin{align*}
\Delta_r(Q) &= \partial \Omega\cap B_r(Q),\,\,\,\qquad T(\Delta_r) = \Omega\cap B_R(Q),\\
\delta(X) &=\text{dist}(X,\partial \Omega).
\end{align*}
\end{definition}

\begin{definition}\label{cmeasure}
Let $T(\Delta_r)$ be the Carleson region associated to a surface
ball $\Delta_r$ in $\delo$, as defined above. A measure $\mu$ in $\Omega$ is
Carleson if there exists a constant $C=C(r_0)$ such that for all
$r\le r_0$,
\[\mu(T(\Delta_r))\le C \sigma (\Delta_r).\]
The best possible $C$ is the Carleson norm and will denoted by $\|\mu\|_{Carl}$. When we want to
emphasize the dependence of $C$ on $r_0$ we shall write $\|\mu\|_{Carl,r_0}$.
When $\mu$ is Carleson we use the notation $\mu \in \mathcal C$.\vglue1mm \noindent If
$\displaystyle\lim_{r_0\to 0} \|\mu\|_{Carl,r_0}=0$, then we say that the
measure $\mu$ satisfies the vanishing Carleson condition, and we
denote this by writing $\mu \in \mathcal C_V$.
\end{definition}

\begin{definition}\label{dgamma}
A cone of aperture $a$ is a non-tangential approach region for $Q
\in \partial \Omega$ of the form
\[\Gamma_{a}(Q)=\{X\in \Omega: |X-Q|\le (1+a)\;\; dist(X,\partial \Omega)\}.\]
Sometimes it will be necessary to truncate $\Gamma_a(Q)$, so we define $\Gamma_{a,h}(Q)= \Gamma_a(Q) \cap B_h(Q)$.
\end{definition}

%
\begin{definition}
If $\Omega \subset \R^n$, the square function of a function $u$ defined on $\Omega$, relative to the family of cones $\{\Gamma_a(Q)\}_{Q \in \partial\Omega}$, is
\[S_{[a]}(u)(Q) =\left( \int_{\Gamma_a(Q)} |\nabla u (X)|^2(X)
dist(X,\partial \Omega)^{2-n} dX \right)^{1/2}\] at each $Q \in \partial
\Omega$. The
non-tangential maximal function relative to $\{\Gamma_a(Q)\}_{Q \in \partial\Omega}$ is
$$N_{[a]}(u)(Q) = \sup_{X\in \Gamma_a(Q)} |u(X)|$$
at each $Q \in \partial\Omega$
The truncation at height $h$ of the non-tangential maximal function is defined by $N_{[a],h}(u)(Q)=\sup_{X\in\Gamma_a(Q)\cap B_h(Q)} |u(X)|$,
with a similar notation $S_{[a],h}$ for truncated square function.\\
It will often be convenient to supress one or both of the parameters $a$ and $h$ in the square and non-tangential functions when their values do not play a significant role in an argument. So we may write $S$, $S_{[a]}$ or $S_h$ to denote $S_{[a],h}$ when no confusion should arise. Similarly we may abreviate $N_{[a],h}$ as $N$, $N_{[a]}$ or $N_h$.\\
We also define the following variant of the non-tangential
maximal function:
\begin{equation}\label{NTMaxVar} \widetilde{N}(u)(Q) = \widetilde{N}_{[a],h}(u)(Q)
=\sup_{X\in\Gamma_{a,h}(Q)}\left(\fint_{B_{{\delta(X)}/{2}}(X)}|u(Y)|^2\diff
Y\right)^{\frac{1}{2}}.
\end{equation}
\end{definition}

\begin{definition} Let $1<p\le\infty$.
The Dirichlet problem with data in $L^p(\partial \Omega, d\sigma)$
is solvable (abbreviated $(D)_{p}$) if for every $f\in C(\partial
\Omega)$ the weak solution $u$ to the problem $Lu=0$ with
continuous boundary data $f$ satisfies the estimate
\[\|N(u)\|_{L^p(\partial \Omega, d\sigma)} \lesssim \|f\|_{L^p(\partial \Omega, d\sigma)}.\]
The implied constant depends only the operator $L$, $p$, and the
Lipschitz character of the domain as measured by the triple
$(\ell,N,C_0)$ of Definition \ref{DefLipDomain}.
\end{definition}

\begin{definition}\label{DefRpcondition}
Let $1<p<\infty$. The regularity problem with boundary data in
$H^{1,p}(\partial\Omega)$ is solvable (abbreviated $(R)_{p}$), if
for every $f\in H^{1,p}(\partial\Omega)\cap C(\partial \Omega)$
the weak solution $u$ to the problem
\begin{align*}
\begin{cases}
Lu &=0 \quad\text{ in } \Omega\\
u|_{\partial B} &= f \quad\text{ on } \partial \Omega
\end{cases}
\end{align*}
satisfies
\begin{align}
\nonumber \quad\|\widetilde{N}(\nabla u)\|_{L^p(\partial
\Omega)}\lesssim \|\nabla_T f\|_{L^{p}(\partial\Omega)}.
\end{align}
Again, the implied constant depends only the operator $L$, $p$,
and the Lipschitz character of the domain.
\end{definition}

\begin{definition}\label{NPDefNpcondition}
Let $1<p<\infty$. The Neumann problem with boundary data in
$L^p(\partial\Omega)$ is solvable (abbreviated $(N)_{p}$), if for
every $f\in L^p(\partial\Omega)\cap C(\partial \Omega)$ such that
$\int_{\dom} fd\sigma=0$ the weak solution $u$ to the problem
\begin{align*}
\begin{cases}
Lu &=0 \quad\text{ in } \Omega\\
A\nabla u\cdot \nu &= f \quad\text{ on }
\partial \Omega
\end{cases}
\end{align*}
satisfies
\begin{align}
\nonumber \quad\|\widetilde{N}(\nabla u)\|_{L^p(\partial
\Omega)}\lesssim \|f\|_{L^{p}(\partial\Omega)}.
\end{align}
Again, the implied constant depends only the operator $L$, $p$,
and the Lipschitz character of the domain. Here $\nu$ is the outer
normal to the boundary $\dom$.
\end{definition}

We are now ready to formulate our main results.

\begin{theorem}\label{LPreg} Let $1<p<\infty$ and let $\Omega\subset {\mathbb R}^n$, $n\ge 2$ be a
bounded Lipschitz domain with Lipschitz constant $\ell$ and $Lu=\mbox{div}(A\nabla u)$ be a uniformly elliptic
differential operator defined on $\Omega$ with ellipticity
constant $\Lambda$ and coefficients which are such that
\begin{equation}\label{carlM}
d\mu=\delta(X)^{-1}\left(\mbox{osc}_{B(X,\delta(X)/2)}a_{ij}\right)^2\,dX
\end{equation}
is the density of a Carleson measure with norm $\|\mu\|_{Carl,r_0}$ on Carleson regions of size at most $r_0$. Then there exists
$\varepsilon=\varepsilon(\Lambda,n,p)>0$ such that if
$\max\{\ell,\|\mu\|_{Carl,r_0}\}<\varepsilon$ then the $(R)_p$ regularity problem
\begin{align*}
\begin{cases}
Lu &=0 \quad\text{ in } \Omega\\
u|_{\partial \Omega} &= f \quad\text{ on } \partial \Omega\\
\widetilde{N}(\nabla u) &\in  L^p(\partial\Omega)
\end{cases}
\end{align*}
is solvable for all $f$ with $\|\nabla_T
f\|_{L^p(\partial\Omega)}<\infty$. Moreover, there exists a
constant $C=C(\Lambda,n,p)>0$ such that
\begin{equation}
\|\widetilde{N}(\nabla u)\|_{L^p(\partial\Omega)}\le C\|\nabla_T
f\|_{L^p(\partial\Omega)}.
\end{equation}

\noindent In particular, if the domain $\Omega$ is $C^1$ and
$A=(a_{ij})$ satisfies the vanishing Carleson condition, then the
regularity problem is solvable for all $1<p<\infty$. More
generally, if the boundary of the domain $\Omega$ is given locally by a function $\phi$ such that $\nabla
\phi$ belongs to $L^{\infty} \cap \mbox{VMO}$, then, once again, the
regularity problem is solvable for all $1<p<\infty$.
\end{theorem}

For the Neumann problem we have an analogous result.

\begin{theorem}\label{NPLPneu} Let $L$, $\Omega$, $p$, $n$, $\Lambda$, $r_0$, $\mu$ have the same meaning and satisfy the same assumptions
as in Theorem \ref{LPreg}.

Then there exists
$\varepsilon=\varepsilon(\Lambda,n,p)>0$ such that if
$\max\{\ell,\|\mu\|_{Carl,r_0}\}<\varepsilon$ then the $(N)_p$ Neumann problem
\begin{align*}
\begin{cases}
Lu &=0 \quad\text{ in } \Omega\\
A\nabla u \cdot \nu &= f \quad\text{ on } \partial \Omega\\
\widetilde{N}(\nabla u) &\in  L^p(\partial\Omega)
\end{cases}
\end{align*}
is solvable for all $f$ in $L^p(\partial\Omega)$  such that
$\int_{\dom} fd\sigma=0$. Moreover, there exists a constant
$C=C(\Lambda,n,p)>0$ such that
\begin{equation}
\|\widetilde{N}(\nabla u)\|_{L^p(\partial\Omega)}\le C\|
f\|_{L^p(\partial\Omega)}.
\end{equation}
\end{theorem}

We first discuss the proof of Theorem \ref{LPreg}.

\begin{proof}  It will follow from Theorem \ref{L2reg2} that the $(R)_2$
regularity problem is solvable for operators satisfying
(\ref{carlM}), provided $\varepsilon$ is sufficiently small. To complete the proof we will use \cite[Theorem 1.1]{DK}. (There is also an older
result by \cite{S} for symmetric operators which should be
adaptable to the non-symmetric case.)

According to \cite[Theorem 1.1]{DK}, $(R)_2$ solvability implies
the solvability of the Regularity problem in an end-point Hardy-Sobolev
space, a boundary-value problem corresponding to $p=1$. 
One can then use  \cite[Theorem 1.1]{DK}  to conclude that for $p\in (1,\infty)$:
\begin{equation}
(R)_p \mbox{ is solvable if and only if } (D^*)_{p'}\mbox{ is
solvable for }p'=p/(p-1).\label{RPDP}
\end{equation}
Here $(D^*)_{p'}$ denotes the $L^{p'}$ Dirichlet problem for the
adjoint operator $L^*u=\div(A^t\nabla u)$.\vglue2mm

However by \cite[Corollary 2.3]{DPP} the $L^{p'}$ Dirichlet
problem for the operator $L^*$ is solvable under the assumptions
of Theorem \ref{LPreg} (for sufficiently small
$\varepsilon=\varepsilon(p')>0$). Hence by (\ref{RPDP}) the
$(R)_p$ problem for the operator $L$ is solvable proving our
claim.
\end{proof}

As follows from the proof given above we also have a result for
the endpoint $p=1$.

\begin{corollary} Under the same assumptions as in
Theorem \ref{LPreg} the $(R)_{HS^1}$ regularity problem for the
operator $L$ is solvable for all $f$ with $\nabla_T f$ in the
atomic Hardy space (c.f.~\cite[Theorem 2.3]{DK}). Moreover, there
exists a constant $C=C(\Lambda,n,a)>0$ such that
\begin{equation}
\|\widetilde{N}(\nabla u)\|_{L^1(\partial\Omega)}\le C\|\nabla_T
f\|_{{\hbar^1(\partial\Omega)}}.
\end{equation}
\end{corollary}

For the Neumann problem a result analogous to the duality (\ref{RPDP}) is unknown and will require a more complicated approach.
We return now to the proof of Theorem \ref{NPLPneu}.

\begin{proof} As follows from Theorem \ref{NPL2neu} the $(N)_p$
Neumann problem is solvable for operators satisfying
(\ref{carl}), provided $\varepsilon$ is sufficiently small and
$p$ is an integer. To replace the condition (\ref{carl}) by
(\ref{carlM}) we use the same idea as \cite[Corollary 2.3]{DPP}
and Theorem \ref{L2reg2}. For a matrix $A$ satisfying
(\ref{carl}) with ellipticity constant $\Lambda$ one can find
(by mollifying the coefficients of $A$) a new matrix $\widetilde{A}$
with same ellipticity constant $\Lambda$ such that $\widetilde{A}$
satisfies (\ref{carlM}) and
\begin{equation}
\sup\{\delta(X)^{-1}|(A-\widetilde{A})(Y)|^2;\, Y\in
B(X,\delta(X)/2)\}\label{NPePert}
\end{equation}
is a Carleson norm. Moreover, if the Carleson norm for matrix $A$
is small (on balls of radius $\le r_0$), so are the Carleson norms
of (\ref{carlM}) for $\widetilde{A}$ and (\ref{NPePert}). Hence
by Theorem \ref{NPL2neu} the $(N)_p$ regularity problem is
solvable for the operator
$\widetilde{L}u=\mbox{div}(\widetilde{A}\nabla u)$.

The solvability of the Neummann problem for perturbed operators
satisfying (\ref{NPePert}) has been studied in \cite{KP2}. It
follows by \cite[Theorem 2.2]{KP2} that the $L^p$ Neumann problem
for the operator $L$ is solvable, provided (\ref{NPePert}) has
small Carleson norm and the regularity $(R)_p$ and Neumann $(N)_p$
problems are solvable for $\widetilde{L}$. Actually, the results
in \cite{KP2} are stated for symmetric operators, however careful
study of the proof of \cite[Theorem 2.2]{KP2} reveals that
symmetry is not necessary.

However by Theorem \ref{LPreg} the $(R)_p$ regularity problem for
$\widetilde{L}$ is solvable provided the Carleson norm of
(\ref{carl}) is sufficiently small and $(N)_p$ Neumann problem
for $\widetilde{L}$ is solvable by Theorem \ref{NPL2neu}. Hence we
have solvability of the Neumann problem $(N)_p$ for $L$ by
\cite[Theorem 2.2]{KP2}.

If $p>1$ is not an integer we use \cite[Theorem 6.2]{KP3}. (This result is also stated for symmetric operators, however, once again, symmetry is not necessary.) This
theorem implies that $(N)_p$ is solvable, provided $(R)_k$ and
$(N)_k$ are solvable, where $k$ is any integer larger than
$p$.
\end{proof}

\section{The Square Function for the Gradient of a Solution}

In this section we shall assume that $\Omega\subset{\mathbb R}^n$ is a smooth bounded domain.
As we shall see, the
case of a Lipschitz domain with small coefficients can be reduced
to this situation via a pull back map of Dahlberg-Kenig-Stein and Ne\v{c}as. (see (\ref{eDKS})).

Because $\Omega$ is bounded, we can think of $\Omega$ as being embedded into a large torus
${\mathbb T}^n$. We aim to establish local results near
$\partial\Omega$. For this reason we introduce a convenient
localization and parametrization of points near $\dom$.

We want to write any point $X\in\Omega$ near $\dom$ as $X=(x,t)$
where $x\in\dom$ and $t>0$. The boundary $\dom$ itself then will
be the set $\{(y,0);y\in\dom\}$. One way to get such a
parametrization is to consider the inner normal $N$ to the
boundary $\dom$. The assumption that $\dom$ is smooth implies
smoothness of $N$. On $\Omega$ we have a smooth underlying metric
of the flat torus ${\mathbb T}^n$.

We consider the geodesic flow ${\mathcal F}_t$ in this metric
starting at any point $x\in\dom$ in the direction $N(x)$. We
assign to a point $X\in\Omega$ coordinates $(x,t)$ if $X={\mathcal
F}_tx$. This means that starting at $x\in\dom$ it takes time $t$
for the flow to get to the point $X$. It's an easy exercise that
the map $(x,t)\mapsto X={\mathcal F}_tx$ is a smooth
diffeomorphism for small $t\le t_0$. Using this parametrization we
consider the set $\Omega_{t_0}=\{(x,t);(x,0)\in\dom\mbox{ and
}0<t<t_0\}$.

Let us now deal with the issue of the metric. We want to work with
the simplest possible metric on $\Omega$ available. Since we only
work on $\Omega_{t_0}$ we take our metric tensor there to be a
product $d\sigma\otimes dt$ where $d\sigma$ is the original metric
tensor on $\Omega$ restricted to $\dom$. The product metric
$d\sigma\otimes dt$ is different from the original metric on
$\Omega$, but they are both smooth and comparable, that is the
distances between points are comparable. Now we express the
operator $L$ in this new product metric.\vglue2mm

We note that under this what is effectively a change of variables, the new coefficients of our
operator are going to satisfy the same Carleson condition as the
original coefficients with Carleson norm comparable to the
original. We observe in particular that the Carleson condition
implies that $\nabla A\in L^\infty_{loc}( \Omega_{t_0})$ hence any
solution of $Lu=0$ on $\Omega_{t_0}$ has a well defined pointwise
gradient $\nabla u$. Furthermore, in the product metric
$d\sigma\otimes dt$, the gradient $\nabla u$ can be written as
$$\nabla u =(\nabla_T u, \partial_t u),$$
where $\nabla_T$ is the gradient restricted to the $n-1$ dimensional
set $\dom\times\{t=\text{const}\}$.\vglue2mm

Frequently throughout the paper it will be useful to localize to a single coordinate patch. The following definition gives a precise notion of coordinate frame.
\begin{definition}\label{DFrame} Let $\dom$ be a smooth $n-1$ dimensional
compact Riemannian manifold. We say that a finite collection of
smooth vector fields $(\vec{T_\tau})_{\tau=1}^m$ ($m\ge n-1) \in T^*(\partial\Omega)$ is a
coordinate frame for $\dom$ if:
\begin{itemize}
\item there is a finite collection of open sets $U_1,U_2,\dots,
U_k$ in $\R^{n-1}$ and smooth diffeomorphisms
$\varphi_s:U_s\to\dom$ such that
$\bigcup_s\varphi_s(\widetilde{U_s})$ covers $\dom$, where
$\widetilde{U_s}$ is an open subset of $U_s$ such that
$\overline{\widetilde{U_s}}\subset U_s$;

\item for each $1\le s \le k$ there exist a set $A_s\subset
\{1,2,\dots,m\}$ such that $|A_s|=n-1$ and
$$\{\varphi_s^*(\vec{T_\tau})\big|_{\widetilde{U_s}};\,\tau\in A_s\}=\{\textstyle\frac{\partial}{\partial x_j}\big|_{\widetilde{U_s}};\,j=1,2,\dots,n-1\}.$$
That is the pullback of the vectors $\vec{T_\tau}$ to $U_s$, $\tau\in
A_s$ restricted to $\widetilde{U_s}$ are just coordinate vector
fields on $\widetilde{U_s}$.
\end{itemize}
\end{definition}

Clearly, $\dom$ has at least one such coordinate frame. Indeed,
the existence of a finite collection
$(U_s,\widetilde{U_s},\varphi_s)$ satisfying all assumptions of
the previous definition follows from the fact that $\dom$ is a
smooth compact Riemannian manifold. Then on each $U_s$ we consider
vector fields $\psi_s\frac{\partial}{\partial_j}$,
$j=1,2,\dots,n-1$ where $\psi_s\in C_0^\infty(U_s)$ is a smooth
cutoff function such that $\psi_s \big|_{\widetilde{U_s}}=1$ and
$0\le \psi_s \le 1$ on $U_s$. Then
$$\{{\varphi_s}_*(\psi_s\textstyle\frac{\partial}{\partial_j}); 1\le s \le k,\, 1\le j\le n-1\}$$
is one such coordinate frame. Here ${\varphi_s}_*$ denotes the
push-forward of a vector field from $U_s$ onto $\dom$.\vglue2mm

We start with the following key lemma for the square function
$S(\nabla_T u)$.

\begin{lemma}\label{l1} Let $u$ be a solution of $Lu=0$, where $L=\div(A\nabla u)$ is a uniformly elliptic differential operator
defined on $\Omega_{t_0}$ with bounded coefficients such
that
\begin{equation}\label{carl}
d\mu=\sup \{ \delta(X)|\nabla a_{ij}(Y)|^2 \;:\;Y\in B_{\delta(X)/2}(X)
\}\,dX
\end{equation}
is the density of a Carleson measure on all Carleson boxes of size
at most $r_0$ with norm $\|\mu\|_{Carl,r_0}$.

Then there exists $r_1>0$ and $K>0$ depending only on the geometry
of the domain $\Omega$, elliptic constant $\Lambda$ and dimension
$n$ such that
\begin{eqnarray}
\label{e1} &&\int_{\partial\Omega}S^2_{r/2}(\nabla_Tu)\,d\sigma
\simeq \iint_{\dom\times (0,r/2)}|\nabla(\nabla_T u(X))|^2
\delta(X)\,dX
\\\nonumber &\le &K\left[\int_{\dom}
|\nabla_T u|^2\,d\sigma+\|\mu\|_{Carl,r_0}\int_{\dom}N_r^2(\nabla u)
d\sigma+\frac{1}{r}\|\nabla u\|^2_{L^2(\Omega)}\right],
\end{eqnarray}
for all $r\le \min\{r_0,r_1,t_0\}$. Here $N_r$ denotes is the non-tangential maximal
function truncated at height $r$, $\delta(X)=t$ for a point
$X=(x,t)\in \dom\times (0,t_0)$, $\nabla_T u$ is the tangential
gradient of $u$ on $\partial\Omega\times\{t\}$ and $dX$ is the
product measure $d\sigma dt$.
\end{lemma}

\begin{proof} In order to establish (\ref{e1}) we localize to
coordinates. Let $U=U_s$ be one of the sets from Definition \ref{DFrame} with
corresponding map $\varphi=\varphi_s$, equally set
$\widetilde{U}=\widetilde{U_s}$. We can now consider the operator
$L$ as being defined on an open subset $U\times(0,t_0)$ of
$\R^{n}_+$, where $\dom$ corresponds to the hyperplane
$\{(x,0);x\in U\}$. We achieve this by pulling back the
coefficients of $L$ from $\Omega_{t_0}$ to $U\times(0,t_0)$ using
the smooth map $\Phi: (x,t)\mapsto (\varphi(x),t)$.\ At this stage
we also pull back the product metric $d\sigma\otimes dt$ from
$\dom\times (0,t_0)$ to $U\times(0,t_0)$ and we get another
product metric that we (in a slight abuse of notation) still
denote by $d\sigma\otimes dt$ on $U\times(0,t_0)$. \vglue2mm

Since we are going to use a partition of unity we also consider a
smooth cutoff function $\phi(x,t)=\phi(x)$ defined on $\R^{n}_+$,
independent of the $t=x_n$ variable such that
\[0\le \phi(x,t)\le
1,\quad \mbox{supp }\phi\subset \widetilde{U} \times \R.\]

Instead of the left-hand side of (\ref{e1}) (by the ellipticity of
the coefficients) we are going to estimate a similar object

\begin{equation}
\iint_{U\times(0,r)}\frac{a_{ij}}{a_{nn}}(\partial_iV)(\partial_jV)\phi
t\, d\sigma\, dt,\label{SpV}
\end{equation}
for functions $V=\nabla u\cdot \vec{T}_\tau$, for $1\le \tau\le m$. Here
and below we use the summation convention and consider the variable
$t$ to be the $n$-th variable. We begin by integrating by parts
\begin{eqnarray}\label{e2}
&&\iint_{U\times
(0,r)}\frac{a_{ij}}{a_{nn}}(\partial_iV)(\partial_jV)\phi t\,
d\sigma\, dt
=\frac12\int_{U\times\{r\}}\partial_j(|V|^2)\frac{a_{ij}}{a_{nn}}\phi
t\nu_i\, d\sigma-\\
\nonumber&-&\iint_{U\times (0,r)}\frac1{a_{nn}}V(LV)\phi t\,
d\sigma\, dt-\iint_{U\times
(0,r)}V(\partial_jV)a_{ij}\partial_i\left(\frac{\phi
t}{a_{nn}}\right)\, d\sigma\, dt.
\end{eqnarray}
Here $\nu_i$ is the $i$-th component of the outer normal $\nu$,
which (given we consider a product metric) is is just the vector
$e_{n}$ for the boundary $U\times\{r\}$. Hence the first term is
non-vanishing only for $i=n$. We work on the last term, as it is
the most complicated. This one splits into three new terms, one
when the derivative hits $t$ (where only the term with $i=n$ will remain)
and another two when it hits $\phi$ and $1/a_{nn}$:

\begin{eqnarray}\nonumber
-\iint_{U\times (0,r)}V(\partial_jV)\frac{a_{nj}}{a_{nn}}\phi \,
d\sigma\, dt&-&\iint_{U\times
(0,r)}V(\partial_jV)\frac{a_{ij}}{a_{nn}}(\partial_i\phi)t\,
d\sigma\, dt\\&+&\iint_{U\times
(0,r)}V(\partial_jV)\frac{a_{ij}}{a^2_{nn}}(\partial_ia_{nn})\phi
t\, d\sigma\, dt .\label{e3}\end{eqnarray} Consider now the first
term of (\ref{e3}). For $j=n$, as $\phi$ is independent of
$x_{n}=t$, we only get

\begin{equation}
-\frac{1}2\iint_{U\times(0,r)}\partial_n(|V|^{2}\phi) \, d\sigma\,
dt=\frac12\int_{U}|V|^2\phi\,
d\sigma-\frac{1}2\int_{U\times\{r\}}|V|^2\phi\,
d\sigma\label{e4}\end{equation}

For $j<n$ the first term of (\ref{e3}) is handled as follows. We
introduce an artificial term $1=\partial_{n} t$ inside the
integral and integrate by parts.
\begin{eqnarray}\nonumber
&-&\frac{1}2\iint_{U\times(0,r)}\partial_j(|V|^{2})\frac{a_{nj}}{a_{nn}}\phi(\partial_{n}t)
\, d\sigma\,
dt=-\frac12\int_{U\times\{r\}}\partial_j(|V|^{2})\frac{a_{nj}}{a_{nn}}\phi
t\, d\sigma\nonumber\\\label{e5}&+&\frac12
\iint_{U\times(0,r)}\partial_{n}\left(\partial_j(|V|^2)\frac{a_{nj}}{a_{nn}}\phi\right)t
\,d\sigma
dt=-\frac12\int_{U\times\{r\}}\partial_j(|V|^{2})\frac{a_{nj}}{a_{nn}}\phi
t\, d\sigma\\&+&
\frac12\iint_{U\times(0,r)}\partial_j\partial_{n}(|V|^2)\frac{a_{nj}}{a_{nn}}\phi
t \, d\sigma
dt+\frac12\iint_{U\times(0,r)}\partial_j(|V|^2)\partial_{n}\left(\frac{a_{nj}}{a_{nn}}\right)\phi
t \, d\sigma dt. \nonumber\end{eqnarray} The first term here gets
completely cancelled out by the first term of (\ref{e2}) as they
have opposite signs. The second term can be further integrated by
parts and we obtain

\begin{eqnarray}\nonumber
\frac12\iint_{U\times(0,r)}\partial_j\partial_{n}(|V|^2)\frac{a_{nj}}{a_{nn}}\phi
t \, d\sigma
dt&=&-\frac12\iint_{U\times(0,r)}\partial_{n}(|V|^2)\partial_j\left(\frac{a_{nj}}{a_{nn}}\right)\phi
t \, d\sigma dt\\&-&\frac12
\iint_{U\times(0,r)}\partial_{n}(|V|^2)\frac{a_{nj}}{a_{nn}}(\partial_j\phi)t
\, d\sigma dt \label{e6}\end{eqnarray}

We now notice that the last term of (\ref{e3}), the third term on
the righthand side of (\ref{e5}) and the first on the righthand
side of (\ref{e6}) are of same type and are bounded from above by

\begin{equation}
 C \iint_{U\times (0,r)}|V||\nabla V||\nabla A|\phi t \,
d\sigma dt.
 \label{e7}\end{equation}
Here $\nabla A$ stands generically for either $\nabla a_{nj}$,
$\nabla a_{nn}$. Estimating (\ref{e7}) further we see that, using
the Cauchy-Schwarz inequality, (\ref{e7}) is less than

\begin{equation}
C\left(\iint_{U\times (0,r)}|V|^{2}|\nabla A|^2\phi t \, d\sigma
dt\right)^{1/2}\left(\iint_{U\times(0,r)}|\nabla V|^2\phi t \,
d\sigma dt\right)^{1/2}.
 \label{e8}\end{equation}
Using the Carleson condition the coefficients satisfy (we drop the dependance of $\|\mu\|_{Carl}$ on $r_0$ for simplicity of the notation) we get that this
can be further written as

\begin{eqnarray}
 \label{e9}&&
C\|\mu\|_{Carl}^{1/2}\left(\int_{U}N_{r}(V)^{2}dy\right)^{1/2}
\left(\iint_{U\times(0,r)}|\nabla V|^2\phi t \, d\sigma
dt\right)^{1/2} \\\nonumber &\le& \frac{\Lambda^2}2
\iint_{U\times(0,r)}|\nabla V|^2\phi t \, d\sigma dt +
\frac{C^2}{2}\|\mu\|_{Carl}\int_{U}N_{r}(V)^{2}dy,
\end{eqnarray}
where the last line follows from the inequality between arithmetic
and geometric means. We observe that the first term on the second
line is no more than one half of (\ref{SpV}) and hence can be
incorporated there. \vglue2mm

Let us summarize what we have. For some constant $C>0$  we have that

\begin{eqnarray}\label{e10}
&&\iint_{U\times(0,r)}\frac{a_{ij}}{a_{nn}}(\partial_iV)(\partial_jV)\phi
t\, d\sigma\, dt\\
\nonumber&\le&C\|\mu\|_{Carl}\int_{U}N_{r}(V)^{2}dy+\int_{U}|V|^2\phi\,
d\sigma-\int_{U\times\{r\}}|V|^2\phi\,
d\sigma+\\\nonumber&+&\int_{U\times\{r\}}\partial_n(|V|^2)\phi t\,
d\sigma -\iint_{U\times (0,r)}\frac1{a_{nn}}V(LV)\phi t\,
d\sigma\, dt +E.
\end{eqnarray}
The fourth term on the righthand side is the first term of
(\ref{e2}) for $i=j=n$.  Here
\begin{equation}\label{eqERR}
E=-\iint_{U\times
(0,r)}\partial_j(|V|^2)\frac{a_{ij}}{a_{nn}}(\partial_i\phi)t\,
d\sigma\, dt-
\iint_{U\times(0,r)}\partial_{n}(|V|^2)\frac{a_{nj}}{a_{nn}}(\partial_j\phi)t
\, d\sigma dt.
\end{equation}
We call $E$ \lq\lq the error terms" these are the second term of
(\ref{e3}) and the second term on the righthand side of
(\ref{e6}). Both terms are of same type and contain $\partial_i
\phi$ for $i<n$. (Recall that $\partial_{n}\phi=0$).\vglue2mm

At this point we have to use the fact that $V=\nabla u\cdot
\vec{T}_i$, for $1\le i\le m$, where $(\vec{T}_i)_{i=1}^m$ is a
frame from Definition \ref{DFrame}. It follows that in our local
coordinates
\begin{equation}
V=\sum_{k<n}b^kv_k,\qquad\text{for some smooth functions $b^k$ on
$U$}.\nonumber
\end{equation}
Here
$$v_k=\partial_k u=\frac{\partial u}{\partial
x_k},\qquad\text{ for }k=1,2,\dots,n-1.$$ We denote by
$v_n=\partial_t u$. We observe that each $v_k$ is a solution of
the following auxiliary inhomogeneous equation:
\begin{equation}
\div(A\nabla v_k)=Lv_k=-\div((\partial_k A){\bf v})=\div
\vec{F_k},\label{eqDer}
\end{equation}
where the $i$-th component of the vector $\vec{F_k}$ is $(\vec{
F_k})^i=-(\partial_ka_{ij})\partial_ju=-(\partial_ka_{ij})v_j$.
\vglue2mm

It remains to deal with the second term of the last line in
(\ref{e10}). Clearly,
\begin{equation}\label{eLV}
LV=\sum_{k<n} \left[\partial_i(a_{ij}(\partial_j
b^k))v_k+a_{ij}(\partial_j b^k)\partial_iv_k+a_{ij}(\partial_i
b^k)\partial_jv_k+b^kLv_k \right].
\end{equation}

We will have to deal with these four terms. We start with the
second and third ones as they are the easiest. We observe that
since $b^i$ are smooth, both $b^i$ and $\nabla b^i$ actually
satisfy the vanishing Carleson condition. Hence these two terms
put into the expression
\begin{equation}\label{eLV2}
\iint_{U\times (0,r)}\frac1{a_{nn}}V(LV)\phi t\, d\sigma\, dt
\end{equation}
can be estimated by
\begin{equation}\label{e23}
 C\sum_{k<n} \iint_{U\times (0,r)}|\nabla u||\nabla v_k||B|\phi t \,
d\sigma dt,
\end{equation}
where $B$ stands for a generic coefficient such as
$a_{ij}(\partial_j b^k)$ or $a_{ij}(\partial_i b^k)$. Observe that $|B|^2t$
is the density of a vanishing Carleson measure, since the $b^k$ are smooth functions. Hence in the same
spirit as we dealt with (\ref{e7}) we get
\begin{eqnarray}&&\label{e20}
\sum_{k<n} \iint_{U\times (0,r)}|\nabla u||\nabla v_k||B|\phi t \,
d\sigma dt\le\\\nonumber &\le& K
\sum_{k<n}\iint_{U\times(0,r)}|\nabla v_k|^2\phi t \, d\sigma dt +
C(K)\|\mu\|_{Carl}\int_{U}N_{r}(\nabla u)^{2}dy.
\end{eqnarray}
We choose $K$ sufficiently small so that the first term
on the second line of (\ref{e20}) can be hidden on the left-hand
side of (\ref{e1}).\vglue2mm

Next we look at the first term of (\ref{eLV}) as we place it into
(\ref{eLV2}). We obtain
\begin{equation}\label{eLV3}
\left|\iint_{U\times
(0,r)}\frac1{a_{nn}}V\partial_i(a_{ij}(\partial_j b^k))v_k\phi t\,
d\sigma\, dt\right|\le C\iint_{U\times (0,r)}|\nabla u|^{2}|\nabla
B|\phi t\, d\sigma\, dt.
\end{equation}
Now we use the fact that $|\nabla B|^2t$ is the density of a Carleson measure with norm $C\|\mu\|_{Carl}$. Hence, by Cauchy-Schwarz,
\begin{eqnarray}\nonumber
\iint_{U\times (0,r)}|\nabla u|^{2}|\nabla B|\phi t\, d\sigma\,
dt&\le& \left(\iint_{U\times (0,r)} |\nabla u|^2t\,d\sigma
dt\right)^{1/2}\left(\iint_{U\times (0,r)} |\nabla u|^2|\nabla
B|^2t\, d\sigma dt\right)^{1/2}\\\label{e41} &\le&
Cr\|\mu\|_{Carl}^{1/2} \int_U N^2_r(\nabla u) d\sigma.
\end{eqnarray}
Here we observe that the last term on the first line is of the
same type as the first term in (\ref{e8}) we have handled before.
\vglue2mm

We now deal with the last term of (\ref{eLV}) using (\ref{eqDer}).
Placing this into (\ref{eLV2}) yields a term
\begin{eqnarray}&&\label{e22}
\sum_{k<n}\iint_{U\times
(0,r)}\frac1{a_{nn}}Vb^k\partial_i((\partial_ka_{ij})v_j)\phi t\,
d\sigma dt=\\\nonumber&=&-\sum_{k<n}\iint_{U\times
(0,r)}\partial_i\left(\frac{b^k\phi}{a_{nn}}Vt
\right)(\partial_ka_{ij})v_j\, d\sigma
dt+\int_{U\times\{r\}}\frac{b^k\partial_ka_{nj}}{a_{nn}}V v_j\phi
t\,d\sigma,
\end{eqnarray}
where we integrate by parts and only obtain a boundary term when
$i=n$. Now we look at the solid integral. This gives
\begin{eqnarray}&\label{e21}
-&\sum_{k<n}\iint_{U\times
(0,r)}\partial_i\left(\frac{b^k}{a_{nn}}\right)(\partial_ka_{ij})Vv_j\phi
t\, d\sigma dt-\\\nonumber&-&\sum_{k<n}\iint_{U\times
(0,r)}\frac{b^k(\partial_ka_{ij})}{a_{nn}}\partial_iV
v_j\,\phi t\, d\sigma dt-\\
\nonumber&-& \sum_{k<n}\iint_{U\times
(0,r)}\frac{b^k\partial_ka_{ij}}{a_{nn}}Vv_j(\partial_i\phi)t\,
d\sigma dt\,-
\\\nonumber&-&\sum_{k<n}\iint_{U\times
(0,r)}\frac{b^k\partial_ka_{nj}}{a_{nn}}V v_j\,\phi\, d\sigma dt,
\end{eqnarray}
where the last term only appears for $i=n$ (as $\partial_n(t)=1$).
We notice that the first term here is of the same type as the
first term in (\ref{e8}) and hence bounded by $C\|\mu\|_{Carl}
\int_{U}N_r(\nabla u)^2\,d\sigma$. The second term is handled
exactly as (\ref{e7}) (noticing that $|v_j|\le |\nabla u|$). Hence
this term is (in absolute value) no greater than
$$K \iint_{U\times(0,r)}|\nabla V|^2\phi t \, d\sigma dt +
C(K)\|\mu\|_{Carl}\int_{U}N_{r}^2(\nabla u)dy,$$ where $K>0$ can be
arbitrary small. Thus as before by choosing $K$ sufficiently
small this term can be absorbed into the left-hand side of
(\ref{e10}).

The third term of (\ref{e21}) is another \lq\lq error" term of
type similar to (\ref{eqERR}). We will handle this at the end. Hence the
only term remaining is
$$-\sum_{k<n}\iint_{U\times
(0,r)}\frac{b^k\partial_ka_{nj}}{a_{nn}}V v_j\,\phi\, d\sigma
dt=-\sum_{k<n}\iint_{U\times
(0,r)}\frac{b^k\partial_ka_{nj}}{a_{nn}}V v_j\,\phi\partial_n(t)\,
d\sigma dt.$$ Here we introduced an extra term $1=\partial_n(t)$
and now integrate by parts again. This gives
\begin{eqnarray}\label{e24}&-&\sum_{k<n}\int_{U\times\{r\}}\frac{b^k\partial_ka_{nj}}{a_{nn}}V
v_j\,\phi t\, d\sigma dt+\\\nonumber&+&\sum_{k<n}\iint_{U\times
(0,r)}\partial_n\left(\frac{b^k}{a_{nn}}\right)(\partial_ka_{nj})V
v_j\,\phi t\, d\sigma dt+
\\\nonumber &+& \sum_{k<n}\iint_{U\times
(0,r)}\frac{b^k\partial_ka_{nj}}{a_{nn}}V
\partial_nv_j\,\phi t\, d\sigma dt+
\\\nonumber &+& \sum_{k<n}\iint_{U\times
(0,r)}\frac{b^k\partial_ka_{nj}}{a_{nn}}\partial_nV v_j\,\phi t\,
d\sigma dt+
\\\nonumber &+&\sum_{k<n}\iint_{U\times
(0,r)}\frac{b^k(\partial_n\partial_ka_{nj})}{a_{nn}}V v_j\,\phi
t\, d\sigma dt.
\end{eqnarray}

The first four terms are of same type we have encountered before.
The first term here is cancelled by the last term of (\ref{e22}).
The second term is bounded by\newline $C\|\mu\|_{Carl} \int_{U}N^2_r(\nabla
u)\,d\sigma$ (c.f. \eqref{e8}). The third term is like (\ref{e23})
and the fourth like the second term of (\ref{e21}). Finally, in
the last term we have two derivatives on the coefficient
($\partial_n\partial_ka_{nj}$) but only one of the derivatives is in
the normal direction since $k<n$. Hence we integrate by parts one
more time (moving the $\partial_k$ derivative). We get
\begin{eqnarray} &&\sum_{k<n}\iint_{U\times
(0,r)}\frac{b^k(\partial_n\partial_ka_{nj})}{a_{nn}}V
v_j\,\phi t\, d\sigma dt=\\
\nonumber &-& \sum_{k<n}\iint_{U\times
(0,r)}\partial_k\left(\frac{b^k}{a_{nn}}\right)(\partial_na_{nj})V
v_j\,\phi t\, d\sigma dt-
\\\nonumber &-& \sum_{k<n}\iint_{U\times
(0,r)}\frac{b^k\partial_na_{nj}}{a_{nn}}V
\partial_kv_j\,\phi t\, d\sigma dt-
\\\nonumber &-&\sum_{k<n}\iint_{U\times
(0,r)}\frac{b^k\partial_na_{nj}}{a_{nn}}\partial_kV
v_j\,\phi t\, d\sigma dt-\\
\nonumber &-& \sum_{k<n}\iint_{U\times
(0,r)}\frac{b^k\partial_na_{nj}}{a_{nn}}V v_j\,(\partial_k \phi)
t\, d\sigma dt.
\end{eqnarray}
Here the second, third and fourth terms are like the second, third and
fourth terms in (\ref{e24}) and are handled likewise. Finally, the
last term is another of the \lq\lq error terms". This concludes
the analysis of the term (\ref{eLV2}) in (\ref{e10}). \vglue2mm

Finally, we sum over all choices of functions $V=V_\tau=\nabla u\cdot
\vec{T}_\tau$, for $1\le \tau\le m$ and over all sets $U_s$ (from
Definition \ref{DFrame}) choosing the smooth cutoff functions
$\phi=\phi_s$ in (\ref{SpV}) such that they are the partition of
unity, that is
$$\sum \phi_s=1\text{ on }\dom\qquad\text{and}\qquad \text{supp }\phi_s\subset\widetilde{U_s}.$$

We first observe that the terms we called \lq\lq error terms"
completely cancel out. There are the terms in (\ref{eqERR}) plus
two extra terms later on. This is due to the fact that  $\sum_s
(\partial_j \phi_s)=0$. That means that summing over $\tau$ these
terms equal to zero. This cancellation happens even if we work
on different coordinate charts since the term we started our
calculation (\ref{SpV}) does not depend on choice of coordinates.
Hence after taking into account all remaining terms we have by
(\ref{e10}):

\begin{eqnarray}\nonumber &&
 \iint_{\dom\times (0,r)}|\nabla(\nabla_T u(X))|^2
\delta(X)\,dX\approx\\\label{e25}&\approx&
\sum_{i=1}^m\iint_{\dom\times (0,r)}|\nabla V_\tau(X)|^2
\delta(X)\,dX\le
\\\nonumber &\le &K\Bigg[\int_{\dom}
|\nabla_T u|^2\,d\sigma+\|\mu\|_{Carl}\int_{\dom}N^2_r(\nabla u)
d\sigma+\\\nonumber&+&\sum_{\tau=1}^m\left[\int_{\dom\times\{r\}}\partial_n(|V_\tau|^2)
r\, d\sigma -\int_{\dom\times\{r\}}|V_\tau|^2\, d\sigma\right]\Bigg].
\end{eqnarray}

At this point we have to deal with the last two terms
\begin{eqnarray}\nonumber &&\int_{\dom\times\{r\}}\partial_n(|V_\tau|^2)
r\, d\sigma -\int_{\dom\times\{r\}}|V_\tau|^2\, d\sigma=\\
\nonumber&=& \int_{\dom\times\{r\}}\partial_n(|V_\tau|^2t) \, d\sigma
-2\int_{\dom\times\{r\}}|V_\tau|^2\, d\sigma\le
\int_{\dom\times\{r\}}\partial_n(|V_\tau|^2t) \, d\sigma.
\end{eqnarray}

We would like to estimate this by a solid integral by integrating
$r$ over an interval $(0,r')$ and averaging. This yields
\begin{equation}\label{e11}
\frac1{r'}\int_0^{r'}\int_{\dom\times\{r\}}\partial_n(|V_\tau|^2t)
\, dX=\int_{\dom\times\{r'\}}|V_\tau|^2\,d\sigma.
\end{equation}

This term is still not a solid integral so we use the averaging
technique one more time by integrating over $r'$ and averaging
over an interval $(0,r_0)$. This yields a solid integral
$$\frac1{r_0}\iint_{\partial\Omega\times(0,r_0)}|V_\tau|^2\,dX\le \frac1{r_0}\iint_{\partial\Omega\times(0,r_0)}|\nabla u|^2\,dX.$$

Going back to (\ref{e25}) we have to perform this double averaging
procedure on all terms. This leads to introduction of some
harmless weight terms and finally an estimate

\begin{eqnarray}
\label{e26} &&  \iint_{\dom\times (0,r/2)}|\nabla(\nabla_T
u(X))|^2 \delta(X)\,dX\le
\\\nonumber &\le &K\left[\int_{\dom}
|\nabla_T u|^2\,d\sigma+\|\mu\|_{Carl}\int_{\dom}N^2_r(\nabla u)
d\sigma+\frac{1}{r}\|\nabla u\|^2_{L^2(\Omega)}\right].
\end{eqnarray}
\end{proof}

Lemma \ref{l1} deals with square function estimates for tangential
directions. We have following for the normal derivative:

\begin{lemma}\label{l2} Under the same assumptions as in Lemma \ref{l1} we have
\begin{eqnarray}
\label{e30} &&\int_{\partial\Omega}S^2_{r}(\partial_nu)\,d\sigma=
\iint_{\dom\times (0,r)}|\nabla(\partial_n u(X))|^2
\delta(X)\,dX\le
\\\nonumber &\le&K\left[\iint_{\dom\times (0,r)}|\nabla(\nabla_T u(X))|^2
\delta(X)\,dX+\|\mu\|_{Carl}\int_{\dom}N_r^2(\nabla
u)\,d\sigma\right]\\\nonumber
&=&K\left[\int_{\partial\Omega}S^2_{r}(\nabla_Tu)\,d\sigma+\|\mu\|_{Carl}\int_{\dom}N_r^2(\nabla
u)\,d\sigma\right]
\end{eqnarray}
provided $r\le \min\{r_0,t_0\}$. Here $\|\mu\|_{Carl}$ is the Carleson norm (\ref{carl})
of the coefficients on Carleson regions of size at most $r_0$ and $K$ only depends on the domain, ellipticity
constant and dimension $n$.
\end{lemma}

\begin{proof} We integrate by parts in
$\partial\Omega\times(0,r)$. We use the notation introduced above
where we denoted $v_n=\partial_n u$. Clearly
\begin{eqnarray}
\label{e31} &&\iint_{\dom\times (0,r)}|\nabla v_n(X)|^2
\delta(X)\,dX\\\nonumber&=&\iint_{\dom\times (0,r)}|\nabla_T
v_n(X)|^2 \delta(X)\,dX+\iint_{\dom\times (0,r)}|\partial_n
v_n(X)|^2 \delta(X)\,dX=\\
\nonumber&=&\iint_{\dom\times (0,r)}|\partial_n (\nabla_T u(X))|^2
\delta(X)\,dX+\iint_{\dom\times (0,r)}|\partial_n v_n(X)|^2
\delta(X)\,dX
\end{eqnarray}
The first term is clearly controlled by the square function of
$\nabla_T u$. It remains to deal with the second term. Since
$$|a_{nn}\partial_nv_n|^2=|\partial_n(a_{nn}v_n)-\partial_n(a_{nn})v_n|^2\le 2|\partial_n(a_{nn}v_n)|^2+2|\partial_n(a_{nn})v_n|^2.$$
We see that by the ellipticity assumption
\begin{eqnarray}
\label{e32} &&\iint_{\dom\times (0,r)}|\partial_n v_n(X)|^2
\delta(X)\,dX\approx\iint_{\dom\times
(0,r)}(a_{nn}(X))^2|\partial_n v_n(X)|^2 \delta(X)\,dX \le
\\\nonumber &\le& 2\iint_{\dom\times (0,r)}|\partial_n(a_{nn}v_n)|^2
t\,dX+2\iint_{\dom\times (0,r)}|\partial_n(a_{nn})v_n|^2 t\,dX.
\end{eqnarray}
Here as before $X=(x,t)$, i.e. $t$ is the last $n$-th coordinate.
The second term (using the Carleson condition) is bounded by
$C\|\mu\|_{Carl}\int_{\dom}N_r^2(\nabla u)\,d\sigma$. We further
estimate the first term. Using the equation $u$ satisfies we see
that
$$\partial_n(a_{nn}v_n)=-\sum_{(i,j)\ne(n,n)}\partial_i(a_{ij}\partial_j u).$$
From this point on we use local coordinates. It follows that
\begin{eqnarray}
\label{e33} &&\iint_{\dom\times (0,r)}|\partial_n(a_{nn}v_n)|^2
t\,dX\le(n^2-1)\sum_{(i,j)\ne(n,n)}\iint_{\dom\times
(0,r)}|\partial_i(a_{ij}\partial_j u)|^2
t\,dX\\\nonumber&\le&2(n^2-1)\sum_{(i,j)\ne(n,n)}\left[\iint_{\dom\times
(0,r)}|\partial_i(a_{ij})|^2|\partial_j u|^2
t\,dX+\iint_{\dom\times (0,r)}|a_{ij}|^2|\partial_i\partial_j
u|^2t\,dX\right].
\end{eqnarray}
The first term here is of the same type as the last term of
(\ref{e32}) and is bounded by $C\|\mu\|_{Carl}\int_{\dom}N_r^2(\nabla
u)\,d\sigma$. Because $(i,j)\ne(n,n)$
$$|\partial_i\partial_j u|^2\le |\nabla(\nabla_T u)|^2,$$
hence the last term of (\ref{e33}) is also bounded by the square
function of $\nabla_T u$.
\end{proof}

\section{Comparability of the Nontangential Maximal Function and the Square Function}

If we combine the results of Lemma \ref{l1} and $\ref{l2}$ we
obtain the following local comparison of the square function and the non-tangential maximal function.

\begin{lemma}\label{l3} Under the same assumption as in Lemma \ref{l1}
there exists constants $r_1>0$ and $K>0$ depending only on the geometry
of the domain $\Omega$, ellipticity constant $\Lambda$, dimension
$n$ and the Carleson norm $\|\mu\|_{Carl}$ of coefficients such that
\begin{eqnarray}
\label{e34} &&\int_{\partial\Omega}S^2_{r/2}(\nabla u)\,d\sigma\le
K\int_{\dom}N_r^2(\nabla u) d\sigma,
\end{eqnarray}
for all $r\le \min\{r_0,r_1,t_0\}$.
\end{lemma}

\begin{proof} We observe that first two terms on the righthand
side of (\ref{e1}) can both be bounded by
$K\int_{\dom}N_r^2(\nabla u) d\sigma$. Recall that the last term
$\frac{K}{r}\|\nabla u\|^2_{L^2(\Omega)}$ appears there due to
averaging of (\ref{e11}). This last averaging is however
unnecessary as the righthand side of (\ref{e11}) can be directly
bounded by a multiple of $\int_{\dom}N_r^2(\nabla u) d\sigma$.
From this (\ref{e34}) follows.
\end{proof}

We would like to establish an analogue of
Lemma \ref{l3} for values $p$ different from $2$. In order to do
that we first observe that a local version of Lemma \ref{l3} is
also true:

\begin{lemma}\label{l3loc} Consider an operator $L$ defined on a
subset $2U\times(0,r)$ of ${\R}^n_+$, with $r \simeq \mbox{diam}(U)$. Then there exists $K>0$
depending only on the ellipticity constant $\Lambda$, dimension
$n$ and the Carleson norm of coefficients such that
\begin{eqnarray}
\label{e34b} &&\int_{U\times (0,r)}|\nabla ^2 u| t\,d\sigma\,dt\le
K\int_{2U}N_r^2(\nabla u) d\sigma.
\end{eqnarray}
\end{lemma}
\begin{proof} The proof is essentially same as the proof of Lemma
\ref{l3} since the estimate (\ref{e1}) is based on local
considerations. However, the terms of type \eqref{eqERR} have to
be considered now as they only disappear in the global estimate.
Observe that $|\partial_i\phi|\le C/r$ hence these \lq\lq error"
terms are bounded from above by
$$C\iint_{U\times(0,r)}|\nabla^2 u||\nabla u|\textstyle\frac{t}{r}d\sigma\,dt.$$
By Cauchy-Schwarz this can be further bounded by
$$C\left(\iint_{U\times(0,r)}|\nabla^2 u|^2t\,d\sigma\,dt\right)^{1/2}\left(\iint_{U\times(0,r)}|\nabla u|^2\textstyle\frac{t}{r^2}d\sigma\,dt\right)^{1/2}.$$
Since $|\nabla u(X)|\le N(\nabla u)(Q)$ for all $X\in \Gamma(Q)$
the term $\iint_{U\times(0,r)}|\nabla
u|^2\textstyle\frac{t}{r^2}d\sigma\,dt$ is further bounded by
$$\int_U{\frac1r}\left(\int_0^r N^2_r(\nabla u)(Q){\textstyle\frac{t}{r}}dt\right)d\sigma(Q)\le \int_U N^2_r(\nabla u)d\sigma.$$
From this (\ref{e34b}) follows.
\end{proof}

We claim that Lemma \ref{l3loc} implies that the square function
is controlled by the non-tangential maximal function in $L^p$ for $p>2$ as
well.

\begin{lemma}\label{lSNp} Under the same assumptions as in Lemma \ref{l1}
for any $p\ge 2$ there exists $r_1>0$ and $K=K(\Omega,\Lambda,n,\|\mu\|_{Carl},p)>0$
such that
\begin{eqnarray}
\label{e34a} &&\int_{\partial\Omega}S^p_{r/2}(\nabla
u)\,d\sigma\le K\int_{\dom}N_r^p(\nabla u) d\sigma,
\end{eqnarray}
for all $r\le \min\{r_0,r_1,t_0\}$.
\end{lemma}
\begin{proof}
The lemma has already been proved when $p = 2$, since then it is just the statement of Lemma \ref{l3}, so we only need to consider $p>2$. Moreover, it suffices to prove \eqref{e34a} on each coordinate patch ${U}_s$ for $s = 1,2,\dots,k$. In fact, we can go slightly further and say it is sufficient to prove
\begin{eqnarray}
\label{e34a2} &&\int_{{U}_s^0}S^p_{r/2}(\nabla
u)\,d\sigma\le K\int_{{U}_s^0}N_r^p(\nabla u) d\sigma
\end{eqnarray}
for each $s$, where $\widetilde{U_s} \subseteq {U}_s^0 \subseteq U_s$. Because we only need to consider $p > 2$, Lemma 2 on page 152 of \cite{St} shows that to prove \eqref{e34a2} it is sufficient to show the relative distributional inequality
\begin{equation} \label{distineq}
\begin{aligned}
& |\{x \in U_s^0 \, | \, S_{[a],r/2}(\nabla u)(x) > 2\lambda, M(N_r(\nabla u)^2)(x)^\frac{1}{2} \leq \alpha\lambda \}| \\
& \leq C\alpha^2 |\{x \in {U}_s^0 \, | \, S_{[2a],r/2}(\nabla u)(x) > \lambda\}|,
\end{aligned}
\end{equation}
where $M$ is the Hardy-Littlewood maximal function on $U_s^0$.

Now, for each $s$, we describe a localised Whitney decomposition of the set where $S_{[2a],r/2}(\nabla u) > \lambda$ (c.f.~\cite[A-34]{G}, which we follow here). First for each $s$ we find a finite number of cubes $Q_{s,j}$ for which $\widetilde{U_s} \subseteq \cup_j Q_{s,j} \subseteq U_s$ and the side length $\ell(Q_{s,j})$ of $Q_{s,j}$ is comparable with $r$. We denote $P_{s,j} = \{x \in Q_{s,j} \, | \, S_{[2a],r/2}(\nabla u)(x) > \lambda\}$ and $K_{s,j} = \{x \in Q_{s,j} \, | \, S_{[2a],r/2}(\nabla u)(x) \leq \lambda\}$

Fix a pair $(s,j)$. If $K_{s,j}$ is empty, define $\mathcal{F}_{s,j} := \{Q_{s,j}\}$. If $K_{s,j}$ is non-empty we will define $\mathcal{F}_{s,j}$ to be a collection of dyadic sub-cubes of $Q_{s,j}$ in the following way. First observe that we can write $P_{s,j}$ as the union of
\[
P_{s,j}^k = \{x \in P_{s,j} \, | \, 2\ell(Q_{s,j})\sqrt{n}2^{-k} < \dist(x,K_{s,j}) \leq 4\ell(Q_{s,j})\sqrt{n}2^{-k}\}
\]
for $k \in \N$.

We can find $2^{n-1}$ dyadic sub-cubes of $Q_{s,j}$ by bisecting each side of $Q_{s,j}$. We denote the collection of these $2^{n-1}$ cubes as $D_{s,j}^1$ and each cube in the collection has side length equal to $\ell(Q_{s,j})/2$. Equally, we can find $2^{n-1}$ dyadic sub-cubes of each cube in $D_{s,j}^1$ by again bisecting each side of it. Thus, we have $2^{2(n-1)}$ subcubes of the cubes in $D_{s,j}^1$ which have $\ell(Q_{s,j})/2^2$. We denote the collection of these $2^{2(n-1)}$ cubes by $D_{s,j}^2$. Continuing inductively $D_{s,j}^k$ is a collection of $2^{k(n-1)}$ dyadic cubes with side length equal to $\ell(Q_{s,j})/2^k$.

Let $\mathcal{F}'_{s,j}$ be the collection of all cubes $Q$ in $D_{s,j}^k$ for some $k \in \N$ such that $Q \cap P_{s,j}^k \neq \emptyset$. Let $Q \in \mathcal{F}'_{s,j}$ and pick $x \in Q \cap P_{s,j}^k$. Observe that
\begin{align*}
s(Q_{s,j})2^{-k}\sqrt{n-1} & = \dist(x,K_{s,j}) - s(Q_{s,j})2^{-k}\sqrt{n-1} = \dist(x,K_{s,j}) - s(Q)\sqrt{n-1} \\
& \leq \dist(Q,K_{s,j}) \leq \dist(x,K_{s,j}) \leq  4s(Q_{s,j})2^{-k}\sqrt{n-1}
\end{align*}
and so,
\begin{equation} \label{whit}
\ell(Q)\sqrt{n-1} \leq \dist(Q,K_{s,j}) \leq  4\ell(Q)\sqrt{n-1}.
\end{equation}
Given that $\mathcal{F}'_{s,j}$ is a collection of dyadic cubes, any two cubes which intersect have the property that one is contained in the other. Thus, we may define $\mathcal{F}_{s,j}$ to be the set of cubes $Q \in \mathcal{F}'_{s,j}$ such that if $Q' \in \mathcal{F}'_{s,j}$ and $Q \cap Q' \neq \emptyset$, then $Q' \subseteq Q$. That is $\mathcal{F}_{s,j}$ is the set of maximal cubes in $\mathcal{F}'_{s,j}$. Clearly then, $\mathcal{F}_{s,j}$ is a collection of disjoint dyadic cubes.

Our Whitney decomposition of $\{x \in U_s \, | \, S_{[a],r/2}(\nabla u)(x) > \lambda\}$ is then the collection
\[
\mathcal{F}_s := \bigcup_{j} \mathcal{F}_{s,j}.
\]
This collection has the properties that each $Q \in \mathcal{F}_s$ is such that either \eqref{whit} holds or $\ell(Q) \simeq r$,
\[
\bigcup_{Q \in \mathcal{F}_s} Q = \{x \in \cup_j Q_{s,j} \, | \, S_{[2a],r/2}(\nabla u)(x) > \lambda\},
\]
and there exists a constant $C$ such that there are at most $C$ cubes that intersect at any given point.

Fix $Q \in \mathcal{F}_s$ and set
\[
R := \{x \in Q \, | \, S_{r/2}(\nabla u)(x) > 2\lambda, M(N_r(\nabla u)^2)(x)^\frac{1}{2} \leq \alpha\lambda \}
\]
If $x \in R$ and \eqref{whit} holds for $Q$, then there exists $x'$ such that $\dist(x,x') \leq 4\sqrt{n}\ell(Q)$ and $S_{[2a],r/2}(\nabla u)(x') \leq \lambda$. Consequently there exists a constant $\alpha$ such that
\begin{equation} \label{two}
\begin{aligned}
S_{[a],\alpha\ell(Q)}(\nabla u)^2(x) & \geq S_{[a],r/2}(\nabla u)^2(x) - \iint_{\Gamma_{[a],r/2}(x) \cap (\R^{n-1} \times (\alpha\ell(Q),r/2))} |\nabla^2 u|^2 t^{2-n}d\sigma dt \\
& \geq S_{[a],r/2}(\nabla u)^2(x) - S_{[2a],r/2}(\nabla u)^2(x') \\
& \geq 4\lambda^2 - \lambda^2 = 3\lambda^2
\end{aligned}
\end{equation}
Then, if $R$ is non-empty (say $x_0 \in R$), we can apply Lemma \ref{l3loc} to conclude that
\begin{equation} \label{one}
\begin{aligned}
|R| & \leq \frac{1}{3\lambda^2} \int_{Q} S_{[a],\alpha\ell(Q)}(\nabla u)^2 d\sigma
\leq \frac{C}{\lambda^2} \int_{Q \times (0,\alpha\ell(Q))} |\nabla^2 u|^2 t \, d\sigma dt \\
& \leq \frac{CK}{\lambda^2} \int_{2Q} N_{\alpha\ell(Q)}(\nabla u)^2 d\sigma
\leq \frac{2^nCK|Q|}{\lambda^2} M(N_r(\nabla u)^2)(x_0)
\leq 2^nCK\alpha^2|Q|.
\end{aligned}
\end{equation}
Furthermore, if \eqref{whit} does not hold, then $r/2 \simeq \ell(Q)$, so we may repeat \eqref{one} without the need for \eqref{two}. Finally, we observe that the inequality $|R| \leq 2^nCK\alpha^2|Q|$ is trivial if $R$ is empty. Thus, summing over $Q \in \mathcal{F}_s$ we obtain \eqref{distineq} with $U_s^0 = \cup_j Q_{s,j}$.
\end{proof}

Now we would like to establish the converse inequality, namely that
the non-tangential maximal function can be dominated by the square
function. As we shall see in the proof we will have to assume {\it
small} Carleson norm $\|\mu\|_{Carl}$ of the coefficients. We start with the following local lemma
working in coordinates on ${\mathbb R}^n_+$ with boundary
${\mathbb R}^{n-1}$.

\begin{lemma}\label{l4} Let $Lu=0$ where $L=\div
A(\nabla \cdot)$ is a uniformly elliptic differential operator
defined on a neighborhood $U$ of $0$ in $\R^n_+$. As before let
(\ref{carl}) be the density of a Carleson measure with norm $\|\mu\|_{Carl}$ on all Carleson
regions of size at most $r_0$.

Let $\phi$ be a non-negative Lipschitz function  and let $Q$ be a
cube in ${\mathbb R}^{n-1}$ with $r=\text{diam}(Q)$. Suppose that
$\phi(x)\le 12r/a$ for $x\in Q^*$. Here $Q^*$ is a dilated $Q$ by
factor of $5$ and $Q^*\times[0,100r]\subset U$. Then if $\|\nabla\phi\|_{L^\infty({\mathbb
R}^{n-1})}$ is sufficiently small, there are exist $a$ (c.f.
Definition \ref{dgamma}) and $C=C(\Lambda,\|\nabla\phi\|_{L^\infty({\mathbb
R}^{n-1})},a)>0$  such that

\begin{eqnarray}
\nonumber \|\nabla u(.,\phi(.))\|^2_{L^2(Q)}\le &C&(\|S(\nabla
u)\|_{L^2(Q^*)}^2+\|\mu\|_{Carl}\|N(\nabla
u)\|^2_{L^2(Q^*)}\\\label{e35}&+&\|N(\nabla
u)\|_{L^2(Q^*)}\|S(\nabla u)\|_{L^2(Q^*)}+r^{n-1}|\nabla
u(X_r)|^2),
\end{eqnarray}
where $X_r$ is an arbitrary corkscrew point, i.e., any point in
$\{X=(x,t); \phi(x)+r/2\le t\le \phi(x)+6r/a\}$. The square and
non-tangential maximal function in (\ref{e35}) are defined using
non-tangential cones $\Gamma_a(.)$. Both square function and
non-tangential maximal functions on the righthand side can be
truncated at a height that is a multiple of $r$.
\end{lemma}

\begin{proof} Recall the mapping $\Phi:{\mathbb
R}^n_+\to\Omega_{\phi}=\{X=(x,t);t>\phi(x)\}$ used by Dahlberg,
Keing and Stein (see for example \cite{D} or \cite{N} and many
others) defined as
\begin{equation}
\Phi(X)=(x,c_0t+(\theta_t*\phi)(x)),\label{eDKS}
\end{equation}
where $(\theta_t)_{t>0}$ is smooth compactly supported approximate
identity and $c_0$ can be chosen large enough (depending only on
$\|\nabla\phi\|_{L^\infty({\mathbb R}^{n-1})}$ so that $\Phi$ is
one to one. We pull back the solution $u$ in $\Omega_{\phi}$ of
div$(A\nabla u)=0$ to a solution $v=u\circ\Phi$ of a different
second order elliptic equation div$(B\nabla v)=0$.

The coefficient matrix $B$ satisfies ellipticity condition with
constant that is a multiple of $\Lambda$ and which depends on
$\|\nabla\phi\|_{L^\infty({\mathbb R}^{n-1})}$. Also if
$\|\mu\|_{Carl}$ is the Carleson norm of
\begin{equation}\nonumber
d\mu=\sup \{ t|\nabla a_{ij}(Y)|^2 \;:\;Y\in B_{t/2}((x,t)) \}dX,
\end{equation}
then the Carleson norm of
\begin{equation}\nonumber
d\mu'=\sup \{ t|\nabla b_{ij}(Y)|^2 \;:\;Y\in B_{t/2}((x,t)) \}dX,
\end{equation}
for $B=(b_{ij})$ will only depend on $\|\mu\|_{Carl}$ and $\|\nabla
\phi\|_{L^\infty}$. Furthermore, if $\|\nabla \phi\|_{L^\infty}$ is small
enough, then the Carleson norm of the matrix $B$ can be guaranteed
to be at most $2\|\mu\|_{Carl}.$

We choose a smooth function $\xi_1:{\mathbb R}^{n-1}\to \mathbb R$
such that $\xi_1(x)=1$ for $x\in Q$, $|\xi_1'|\le16/r$ and support
contained in a concentric dilation $(9/8)Q$. Choose another
function $\xi_2:[0,\infty)\to \mathbb R$ such that $\xi_2(t)=1$ on
$[0,r]$, $|\xi_2'|\le 5/r$ and support contained in $[0,2r]$. Now
define $\xi(X)=\xi(x,t)=\xi_1(x)\xi_2(t)$.

Denote by $w_i=\partial_iv$ for $i=1,2,\dots,n$. For each $i\le
n-1$ we have
\begin{eqnarray}
\nonumber \int_{{\mathbb
R}^{n-1}}w_i(x,0)^2\xi_1(x)dx&=&-\iint_{{\mathbb
R}^{n}_+}\partial_n(w_i^2\xi)(X)\,dX\\
\label{e36} &=&-\iint_{{\mathbb
R}^{n}_+}2w_i(\partial_nw_i)\xi\,dX-\iint_{{\mathbb
R}^{n}_+}w_i^2\xi_1\xi_2'\,dX.
\end{eqnarray}
The second term on the right-hand side of (\ref{e36}) is
controlled by $r^{-1}\iint_Kw_i^2$ where $K=\{X=(x,t);x\in
Q^*\text{ and }r/3\le t\le 7r/a\}$. Let $X_r$ be any point in $K$
and choose $K'$ and $K''$ to be the appropriate concentric
enlargements of $K$. We set $c=\frac1{K'}\iint_{K'}w_i$. Using
\cite[Thm 8.17]{GT} we may further estimate this term by

\begin{eqnarray}
\nonumber&& r^{-1}\iint_{K} (w_i - w_i(X_r))^2 \, dX +
r^{-1}\iint_{K} w_i^2(X_r) \,
dX \\
\nonumber&\le& Cr^{n-1}\, \text{osc}_{K}(w_i)^2 + Cr^{n-1}|w_i(X_r)|^2 \\
\nonumber&\le& Cr^{n-1} \sup_{K}|w_i-c|^2 + Cr^{n-1}|w_i(X_r)|^2 \\
\nonumber&\le& Cr^{-1}\iint_{K'} |w_i-c|^2\, dX +
Cr^{n-1+2(1-n/q)}\|(\partial_i B){\bf w}\|_{L^q(K')}^2 +
Cr^{n-1}|w_i(X_r)|^2,
\end{eqnarray}
for $q>n$. Here we are using (\ref{eqDer}) with matrix $A$
replaced by $B$. Using Poincar\'e's inequality and the Carleson
condition for $B$ this can be further estimated by
\begin{eqnarray}
\nonumber&& C\left[\iint_{K''} |\nabla (w_i)|^2r\, dX + \|\mu\|_{Carl}
\|N(\nabla u)\|_{L^2(Q^*)}^2 +
r^{n-1}|w_i(X_r)|^2\right] \\
\label{e39}&\le&  C\left[ \|S(\nabla u)\|_{L^2(Q^*)}^2 +
\|\mu\|_{Carl}\|N(\nabla u)\|_{L^2(Q^*)}^2 + r^{n-1}|w_i(X_r)|^2\right].
\end{eqnarray}

The first term on the righthand side of (\ref{e36}) can be estimated by
\begin{eqnarray}
\nonumber &-&\iint_{{\mathbb R}^{n}_+}2w_i(\partial_nw_i)\xi\,dX\\
\nonumber& =& -\iint_{\R^n_+} {2w_i (\partial_nw_i)\xi}
\partial_n(t) \, dX = 2\iint_{\R^n_+} [\partial_n({w_i
(\partial_nw_i)\xi})]
t \, dX \\
\nonumber& =& 2\iint_{\R^n_+} (\partial_nw_i)^2\xi {t \,dX} +
2\iint_{\R^n_+} w_i (\partial^2_nw_i)\xi {t \, dX}  +
2\iint_{\R^n_+} w_i (\partial_nw_i)\xi_1\xi_2'
{t \, dX} \\
\nonumber&& =: \mbox{I} + \mbox{II} + \mbox{III}.
\end{eqnarray}
Using the fact that $i\le n-1$ we see that $\partial^2_nw_i$ in
the term II can be written as $\partial_i\partial_nw_n$. This
gives
\begin{eqnarray}
\nonumber
\mbox{II}&=&-2\iint_{\R^n_+}(\partial_nw_n)\partial_i(w_i\xi)t\,dX\\
\nonumber&=&-2\iint_{\R^n_+}(\partial_nw_n)\partial_i(w_i)\xi
t\,dX-2\iint_{\R^n_+}(\partial_nw_n)w_i\partial_i(\xi_1)\xi_2t\,dX\\
\nonumber&=& \mbox{II}_1 + \mbox{II}_2.
\end{eqnarray}

We observe that the terms I and $\mbox{II}_1$ are both bounded by
the square function $\|S_{2r}({\bf w})\|_{L^2(Q^*)}^2$. This is further bounded by $\|S(\nabla
u)\|_{L^2(Q^*)}^2$, where the square function is truncated at a greater height or not truncated at all. For $\mbox{II}_2$
and III we have
\begin{eqnarray}
\nonumber \mbox{II}_2+\mbox{III}&\le&
\frac{C}r\iint_{Q^*\times(0,2r)}|\nabla {\bf w}||{\bf w}| t\,dX\\
\nonumber&\le& C\left(\iint_{Q^*\times(0,2r)}|\nabla {\bf w}|^2
t\,dX\right)^{1/2}\left(\iint_{Q^*\times(0,2r)}|{\bf w}|^2
\textstyle\frac{t}{r^2}\,dX\right)^{1/2}\\
\nonumber&\le& C\|S({\bf
w})\|_{L^2(Q^*)}\left(\int_{Q^*}\frac{1}{r}\int_0^{2r}|{\bf w}|^2
\,dt\,dx\right)^{1/2}\\\nonumber&\le& C\|S({\bf
w})\|_{L^2(Q^*)}\left(\int_{Q^*}\frac{2r}{r}|N({\bf w})|^2
\,dx\right)^{1/2}=C\|S({\bf w})\|_{L^2(Q^*)}\|N({\bf
w})\|_{L^2(Q^*)}.
\end{eqnarray}
This bounds (\ref{e36}) by terms that appear on the righthand side
of (\ref{e35}).\vglue2mm

It remains to estimate $\int_{{\mathbb
R}^{n-1}}w_n(x,0)^2\xi_1(x)dx$. We estimate instead an expression
for co-normal derivative $H=\sum_{j}b_{nj}w_j$. This is sufficient
since
\begin{eqnarray}
\nonumber &&\int_{{\mathbb R}^{n-1}}w_n(x,0)^2\xi_1(x)dx\approx
\int_{{\mathbb
R}^{n-1}}(b_{nn}w_n)^2(x,0)\xi_1(x)dx\\
\label{e37} &\le& n\left[\int_{{\mathbb
R}^{n-1}}H^2\xi_1\,dx+\sum_{j<n}\int_{{\mathbb
R}^{n-1}}(b_{nj}w_j)^2\xi_1\,dx\right]\\
\nonumber &\le& n\int_{{\mathbb
R}^{n-1}}H^2\xi_1\,dx+C\sum_{j<n}\int_{{\mathbb
R}^{n-1}}w_j^2(x,0)\xi_1(x)\,dx
\end{eqnarray}
Hence if we can obtain estimates for the first term we are done
since the second term has already been bounded. We proceed as
before.
\begin{eqnarray}
\nonumber \int_{{\mathbb
R}^{n-1}}H(x,0)^2\xi_1(x)dx&=&-\iint_{{\mathbb
R}^{n}_+}\partial_n(H^2\xi)(X)\,dX\\
\label{e38} &=&-\iint_{{\mathbb
R}^{n}_+}2H(\partial_nH)\xi\,dX-\iint_{{\mathbb
R}^{n}_+}H^2\xi_1\xi_2'\,dX.
\end{eqnarray}
As before we observe that the second term can be bounded by
$r^{-1}\sum_{i}\iint_Kw_i^2$. The calculation we have done above
holds for any $i$ even $i=n$ giving us bound (\ref{e39}).

It remains to deal with the first term. Using the equation
div$(B\nabla v)=0$
$$\partial_nH=\sum_{j}\partial_n(b_{nj}\partial_jv)=-\sum_{i<n}\partial_i(b_{ij}\partial_jv)=-\sum_{i<n}\partial_i(b_{ij}w_j).$$
It follows that
\begin{eqnarray}
\nonumber &-&\iint_{{\mathbb
R}^{n}_+}2H(\partial_nH)\xi\,dX\\\label{e40}&=&\sum_{i<n}\iint_{{\mathbb
R}^{n}_+}2H\partial_i(b_{ij}w_j)(\partial_nt)\xi\,dX=-\sum_{i<n}\iint_{{\mathbb
R}^{n}_+}2\partial_n(H\partial_i(b_{ij}w_j)\xi)t\,dX\\
\nonumber&=&-\sum_{i<n}\iint_{{\mathbb
R}^{n}_+}2(\partial_nH)\partial_i(b_{ij}w_j)\xi
t\,dX-\iint_{{\mathbb
R}^{n}_+}2H\partial_i\partial_n(b_{ij}w_j)\xi
t\,dX\\\nonumber&&\hskip1mm-\iint_{{\mathbb
R}^{n}_+}2H\partial_i(b_{ij}w_j)\xi_1\xi_2' t\,dX =
\widetilde{\mbox{I}}+\widetilde{\mbox{II}}+\widetilde{\mbox{III}}.
\end{eqnarray}
As before we do further integration by parts for the term
$\widetilde{\mbox{II}}$.
\begin{eqnarray}
\nonumber \widetilde{\mbox{II}}&=&\iint_{{\mathbb
R}^{n}_+}2\partial_n(b_{ij}w_j)\partial_i(H\xi)
t\,dX\\
\nonumber&=&2\iint_{{\mathbb
R}^{n}_+}2\partial_n(b_{ij}w_j)(\partial_iH)\xi
t\,dX+\iint_{{\mathbb
R}^{n}_+}2\partial_n(b_{ij}w_j)H(\partial_i\xi_1)\xi_2
t\,dX\\
\nonumber&=& \widetilde{\mbox{II}_1} + \widetilde{\mbox{II}_2}.
\end{eqnarray}
We observe that when the derivative in terms $\widetilde{\mbox{II}_2}$ and $\widetilde{\mbox{III}}$ does not
hit the coefficients $b_{ij}$ these can be estimated exactly as the
corresponding terms $\mbox{II}_2$ and III. When
the derivative falls on the coefficient we get \lq\lq error terms"
that can be estimated using the Carleson measure property of the coefficients.
In particular the term from $\widetilde{\mbox{III}}$ is of the same form as (\ref{e41}) and is handled analogously.
The term we obtain
from $\widetilde{\mbox{II}_2}$ is of a different nature and can be
bounded above by
$$\iint_{Q^*\times(0,2r)}|{\bf w}|^2|\nabla B|\textstyle{\frac{t}r}\,dX.$$
By Cauchy-Schwarz this is no more than
\begin{eqnarray}
\nonumber& & C\left(\iint_{Q^*\times(0,2r)}|\nabla B|^2|{\bf
w}|^2 t\,dX\right)^{1/2}\left(\iint_{Q^*\times(0,2r)}|{\bf w}|^2
\textstyle\frac{t}{r^2}\,dX\right)^{1/2}\\
\nonumber&\le& C\|\mu\|_{Carl}^{1/2}\|N({\bf
w})\|_{L^2(Q^*)}\left(\int_{Q^*}\frac{1}{r}\int_0^{2r}|{\bf w}|^2
\,dt\,dx\right)^{1/2}\\\nonumber&\le& C\|\mu\|_{Carl}^{1/2}\|N({\bf
w})\|_{L^2(Q^*)}\left(\int_{Q^*}\frac{2r}{r}|N({\bf w})|^2
\,dx\right)^{1/2}=C\|\mu\|_{Carl}^{1/2}\|N({\bf w})\|^2_{L^2(Q^*)}.
\end{eqnarray}

The terms $\widetilde{\mbox{I}}$ and $\widetilde{\mbox{II}_1}$ contain both a derivative acting on $H$ and a derivative acting on $b_{ij}w_j$. We deal with these in two parts: (a) when the derivative acting on $H = \sum b_{nj}w_j$ falls on $b_{nj}$ and (b) when it falls on $w_j$. First we deal with case (b). When the derivative acting on $b_{ij}w_j$ does not hit the coefficients, we can handle them as the corresponding terms I and $\mbox{II}_1$. When this derivative falls on the coefficients, the term we get from $\widetilde{\mbox{I}}$ is again of the same nature as (\ref{e41}) and the term we get from $\widetilde{\mbox{II}_1}$ looks like (\ref{e7}), so these terms are handled as before. Finally we deal with case (a), where we either get terms of the form (\ref{e7}), which we have dealt with before, or terms of the form
\begin{equation} \label{form1}
\iint_{Q^*\times(0,2r)} |{\bf w}|^2|\nabla B|^2 t \xi dX \lesssim \|\mu\|_{Carl} \int_{Q^*} N_r^2(\nabla u) d\sigma.
\end{equation}
This concludes the proof as
$$\|\nabla u(.,\phi(.))\|^2_{L^2(Q)}\le C\sum_{i=1}^n\int_{{\mathbb R}^{n-1}}w_i(x,0)^2\xi_1(x)dx.$$
\end{proof}

From now on we follow the stopping time argument from \cite{KKPT},
in particular our Lemma \ref{l4} is an analogue of \cite[Lemma 3.8]{KKPT}. For any continuous function ${\bf v}:\R^n_+\to \R^n$
and $\nu\in \R$ we define
$$h_{\nu,a}({\bf v})(x)=\sup\{t\ge 0;\, \sup_{\Gamma_a(x,t)}|{\bf v}|>\nu\}.$$
Here $\Gamma_a(x,t)$ is a cone with vertex at $(x,t)$ (recall that
the boundary point is $(x,0)$). Hence
$$\Gamma_a(x,t)=(0,t)+\Gamma_a(x,0),$$
is the non-tangential cone $\Gamma_a(x,0)$ shifted in the direction
$(0,t)$.

\begin{lemma}\label{l5} If ${\bf v}$ is such that
$h_{\nu,a}({\bf v})<\infty$ then $h_{\nu,a}({\bf v})$ is Lipschitz
with constant $1/a$.
\end{lemma}

\begin{proof} See, for example \cite[Lemma 3.13]{KKPT}.
\end{proof}

We also have an analogue of \cite[3.14]{KKPT}.

\begin{lemma}\label{l6} Under same assumptions on $u$ and $L$ as in Lemma \ref{l4}
set ${\bf v}=\nabla u$ and let $(Q_j)_j$ be a Whitney decomposition of
$\{x;\, N_{[a]}({\bf v})(x)>\nu/24\}$. Given $a>0$, let $E^j_{\nu,\rho}$
be the intersection of the cube $Q_j$ with
$$\{x;\, N_{[a/12]}({\bf v})(x)>\nu\mbox{ and } \|\mu\|^{1/2}_{Carl} N_{[a]}({\bf v})(x)+S_{[a]}({\bf v})(x)\le\rho\nu\}.$$
Then there exist a sufficiently small choice of $\rho$,
independent of $Q_j$ so that for each $x\in E^j_{\nu,\rho}$ there
is a cube $R$ with $x\in 6R$ and $R\subset Q^*_j$ for which
$$|{\bf v}(z,h_{\nu,a/12}({\bf v})(z))|>\nu/2$$
for all $z\in R$.
\end{lemma}

\begin{proof} Let $x\in E^j_{\nu,\rho}$. By definition $h_{\nu,a/12}({\bf
v})(x)>0$ and so there exists a $Y$ on
$\partial\Gamma_{a/12}(x,h_{\nu,a/12}({\bf v})(x))$ such that
$|{\bf v}(Y)|=\nu$ (here $Y=(y,y_n)$) and $h_{\nu,a/12}({\bf
v})(y)=y_n$. Let $r_0=y_n$ and
$$K=\Gamma_{a/12}(x,0)\cap\{Z;\,|z_n-y_n|<r_0/6\}.$$
Since $Q_j$ is a Whitney cube, $r_0\le (1+4\sqrt{n-1})\ell(Q_j)/a$, and we also have
$$3K\subset \Gamma_a(x,0) \quad \mbox{and} \quad \mbox{dist}(3K,\partial\R^n_+) \geq r_0/2.$$
Hence again by \cite[Thm 8.17]{GT} we have that

\begin{eqnarray}
\nonumber&& \mbox{osc}_{K}({\bf v})\le C(r_0^{-n/2}\|{\bf v}-{\bf
c}\|_{L^2(2K)}+r_0^{1-n/q}\|(\nabla A){\bf v}\|_{L^q(2K)}),
\end{eqnarray}
for any constant ${\bf c}$ and $q>n$. By (\ref{carl}) $|(\nabla
A){\bf v}|(Z)\le Cr_0^{-1}\|\mu\|^{1/2}_{Carl} N_{[a]}({\bf v})(x)$ for $Z\in
2K$, so
$$r_0^{1-n/q}\|(\nabla A){\bf v}\|_{L^q(2K)}\le C\|\mu\|^{1/2}_{Carl} N_{[a]}({\bf v})(x)$$
and so using Poincar\'e's inequality
\begin{eqnarray}
\nonumber |{\bf v}(Z)-{\bf v}(Y)|&\le& \mbox{osc}_{K}({\bf v})\le
C(r_0^{1-n/2}\|\nabla{\bf v}\|_{L^2(3K)}+\|\mu\|^{1/2}_{Carl} N_{[a]}({\bf
v})(x))
\end{eqnarray}
\begin{eqnarray} \nonumber &\le& C(S_{[a]}({\bf v})(x)+\|\mu\|^{1/2}_{Carl}
N_{[a]}({\bf v})(x))\le C\rho\nu,
\end{eqnarray}
for any $Z\in K$. Thus we may choose $\rho$ sufficiently small so
that $|{\bf v}(Z)-{\bf v}(Y)|\le \nu/2$. Then clearly $|{\bf
v}(z,h_{\nu,a/12}({\bf v})(z))|>\nu/2$ for $|z-y|\le ar_0/72$.
\end{proof}

%
%
%
%
%
Finally, the results of this section can be converted to the following result.

\begin{lemma}\label{NPlNS} Under the same assumption as in Lemma \ref{l1}
there exists $\varepsilon>0$ depending only on the geometry
of the domain $\Omega$, the ellipticity constant $\Lambda$, dimension $n$
and $p$ such that if $\|\mu\|_{Carl}<\varepsilon$ then
\begin{eqnarray}
\label{e42} &&\int_{\partial\Omega}N^p_{r/2}(\nabla u)\,dx\le
K\int_{\dom}S_r^p(\nabla u)
dx+\iint_{\Omega\setminus\Omega_{r/2}}|\nabla u|^p\,dX.
\end{eqnarray}
Here $K=K(\Omega,\Lambda,p,n)>0$. $N_h$ and
$S_h$ are truncated versions of non-tangential maximal function
and square function, respectively.
\end{lemma}

\noindent{\it Remark.} The term
$\iint_{\Omega\setminus\Omega_{r/2}}|\nabla u|^p\,dX$ is necessary
if $\Omega$ is a bounded domain. Consider for example $L=\Delta$
on $\Omega\subset\R^n$. Let $u$ be a harmonic function in
$\Omega$. Then for any vector ${\bf c}$ we have that $S(\nabla
u)=S(\nabla (u+{\bf c}\cdot{\bf x}))$ but clearly $N(\nabla u)\ne
N(\nabla (u+{\bf c}\cdot{\bf x}))$. This term is not necessary if
the domain is unbounded and we consider untruncated versions of
the non-tangential maximal function and the square
function.\vglue2mm

\begin{proof} We only highlight the major points of the proof as
the basic idea is the same as in \cite{KKPT}. Applying standard techniques as in \cite[Lemma 3.15]{KKPT}
the stopping time function $h$, Lemmas \ref{l4} and \ref{l6} can be combined into the following good-$\lambda$
inequality.
\[\begin{split}
&\sigma\big(\{x':\, {N}_{[a/12]}(\nabla u)>\nu,\, (M(S^2_{[a]}(\nabla u)))^{1/2}\le\gamma\nu,\,(M(S^2_{[a]}(\nabla u))M({N}_{[a]}^2(\nabla u)))^{1/4}\le\gamma\nu,\\&\quad
(M(\|\mu\|_{Carl}N^2_{[a]}(\nabla u)))^{1/2}\le\gamma\nu
\}\big)\le C(\gamma)\sigma\left(\{x':\, {N}_{[a/12]}(\nabla u)>\nu/32\}\right),
\end{split}\]
for all $\gamma<1$ with $C(\gamma)\to 0$ as $\gamma\to 0$.

Note that Lemma \ref{l4} requires the Lipschitz function $\phi$ to
have a small Lipschitz norm. Since we are using the function
$h_{\nu,a/12}({\bf v})$ in place of $\phi$, if we choose $a>0$
large enough by Lemma \ref{l5} the Lipschitz norm will be small. 

Having the good-$\lambda$ inequality \eqref{e42} follows for $p>2$ immediately by a standard argument (see the discussion above Theorem 3.18 of \cite{KKPT}). Seemingly the term $(M(\|\mu\|_{Carl}N^2_a(u)))^{1/2}\le\gamma\nu$ in the good-$\lambda$ might be problematic, when converting the inequality into 
\eqref{e42} for $p>2$. However, what saves the days is the fact that this term will contribute a factor
$\|\mu\|_{Carl}\|N(\nabla u)\|^p_{L^p}$ which when $\|\mu\|_{Carl}$ is small can be absorbed in the estimate. 

Furthermore as in \cite[Theorem 3.18]{KKPT}  the global result for $p>2$ implies a local version of the  estimate \eqref{e42} also holds for some $p>p_0$. Finally the local estimate for all $p>1$ then follows by a standard argument from the local one for some $r>p_0$. See \cite{FSt} for full details. 
\end{proof}
\section{The $(R)_2$ Regularity Problem}

\begin{theorem}\label{L2reg} Let $\Omega\subset {\mathbb R}^n$ be a
bounded Lipschitz domain with Lipschitz norm $\ell$ and $L=\mbox{div}(A\nabla \cdot)$ be an uniformly elliptic
differential operator defined on $\Omega$ with ellipticity
constant $\Lambda$ and coefficients such that (\ref{carl})
is a Carleson measure with norm $\|\mu\|_{Carl,r_0}$ on Carleson regions of size at most $r_0$. Then there exists
$\varepsilon=\varepsilon(\Lambda,n)>0$ such that if
$\max\{\ell,\|\mu\|_{Carl,r_0}\}<\varepsilon$ then the regularity problem
\begin{eqnarray}
\nonumber &&Lu=0,\qquad\mbox{in }\Omega,\\
\nonumber &&u=f,\qquad\hskip2mm\mbox{on }\partial\Omega,\\
\nonumber &&N(\nabla u)\in L^2(\partial\Omega),
\end{eqnarray}
is solvable for all $f$ with $\|\nabla_T
f\|_{L^2(\partial\Omega)}<\infty$. Moreover, there exists a
constant $C=C(\Lambda,n,a)>0$ such that
\begin{equation}
\|N(\nabla u)\|_{L^2(\partial\Omega)}\le C\|\nabla_T
f\|_{L^2(\partial\Omega)}.\label{mest}
\end{equation}
\end{theorem}

\begin{proof} For any $f$ in the Besov space
$B^{2,2}_{1/2}(\partial\Omega)$ the exists a unique
$H_1^{2}(\Omega)$ solution by the Lax-Milgram theorem. Observe
that $f\in H_1^{2}(\partial\Omega)\subset
B^{2,2}_{1/2}(\partial\Omega)$ so it only remains to establish the
estimate (\ref{mest}).

Consider $\varepsilon>0$ and take $\|\mu\|_{Carl,r_0}<\varepsilon$. From now on we drop the subscript $r_0$. To keep matters simple let us first consider the case when
$\partial\Omega$ is smooth. In this case Lemma \ref{l1} 
applies directly. If follows that for all small $r$
\begin{eqnarray} \label{e44}
&&\int_{\partial\Omega}S^2_{r/2}(\nabla u)\,d\sigma\le
K\left[\int_{\dom} |\nabla_T
u|^2\,d\sigma+\|\mu\|_{Carl}\int_{\dom}N_r^2(\nabla u)
d\sigma+\frac{1}{r}\|\nabla u\|^2_{L^2(\Omega)}\right].
\end{eqnarray}
We now choose $\varepsilon$ small enough such that Lemma \ref{NPlNS} holds. It
follows that by (\ref{e42})
\begin{eqnarray} \nonumber
\int_{\partial\Omega}N_{r/4}^2(\nabla u)\,d\sigma\le
\widetilde{K}\left[\int_{\dom} |\nabla_T
u|^2\,d\sigma+\|\mu\|_{Carl}\int_{\dom}N_r^2(\nabla u)
d\sigma+\frac{1}{r}\|\nabla u\|^2_{L^2(\Omega)}\right].
\end{eqnarray}

We also observe that we have a pointwise estimate
\begin{equation}\label{e46}
N_{r}^2(\nabla u)(X)\le N_{r/4}^2(\nabla
u)(X)+C(r)\iint_{\Omega_{r/8}}|\nabla u(Y)|^2\,dY
\end{equation}
for all $X\in \partial\Omega$. This is easy as we are
estimating $|\nabla u|$ away from the boundary. Hence, by the
Carleson condition we have $|\nabla A|\le \|\mu\|^{1/2}_{Carl}/r$ there. A
standard bootstrap argument using the fact that ${\bf v}=\nabla u$
satisfies the equation $L{\bf v}=\div((\nabla A){\bf v})$ 
yields pointwise bounds on $|\nabla u|$ for
$\{X\in\partial\Omega;\,\mbox{dist}(X,\partial\Omega)\in
[r/4,r]\}$. Finally, using (\ref{e46}) we
obtain
\begin{eqnarray} \label{e47}
\int_{\partial\Omega}N_{r}^2(\nabla u)\,d\sigma\le
\widetilde{K}\Bigg[\int_{\dom} |\nabla_T
u|^2\,d\sigma&+&\|\mu\|_{Carl}\int_{\dom}N_r^2(\nabla u)
d\sigma\Bigg]\\\nonumber&+&C(r)\|\nabla
u\|^2_{L^2(\Omega)}.
\end{eqnarray}
We now can make our final choice of $\varepsilon$. We choose it
sufficiently small such that the constant in (\ref{e47})
$K\|\mu\|_{Carl}<1/2$ which yields
\begin{eqnarray} \label{e48}
\int_{\partial\Omega}N_{r}^2(\nabla u)\,d\sigma\le
2\widetilde{K}\int_{\dom} |\nabla_T
u|^2\,d\sigma+2C(r)\|\nabla
u\|^2_{L^2(\Omega)}.&
\end{eqnarray}
From this the desired estimate follows since the term $\|\nabla
u\|^2_{L^2(\Omega)}$, i.e., an $H^2_1(\Omega)$ estimate of the
solution $u$, follows from Lax-Milgram.\vglue2mm

Now we turn to the more general case, when $\Omega$ has a Lipschitz
boundary with sufficiently small Lipschitz constant $\ell$. This case
also includes the $C^1$ boundary as in such case $\ell$ can be taken
arbitrary small.

The crucial point is that the proofs of
Lemmas~\ref{l1}-\ref{NPlNS} in the smooth case are based on local
estimates near boundary $\partial\Omega$. 
We refer to \cite{BZ}, in particular Theorem 5.1 and Remark 5.3, for the construction of approximations of
Lipschitz domains $\Omega$ by smooth domains $\Omega_{\epsilon}$ via bi-Lipschitz homeomorphisms where the Lipschitz constant is independent of $\epsilon$.
This transforms the original equation on $\Omega$ to a new elliptic equation on $\Omega_{\epsilon}$ with coefficients satisfying a Carleson condition of the same order
of magnitude.
\end{proof}

\noindent {\it Remark }   We claim that the assumption that the domain has a Lipschitz boundary with small Lipschitz constant can be replaced by the assumption that the boundary
is given locally by a function whose gradient has small BMO norm.  
If the boundary locally coincides with $\{(x,t)\in\R^{n}; t>\phi(x)\}$, we use the fact that 
the map $\Phi$ in  (\ref{eDKS}) is a bijection between the sets
$\R_+^{n}$ and $\{(x,t)\in\R^{n}; t>\phi(x)\}$ provided that $c$ is chosen to be larger that  $ \|\nabla\phi\|_{BMO}$.
Hence by pulling back everything
(metric, coefficients) using $\Phi$ we are left with proving local
estimates on a subset of $\R_+^{n}$. We now have to estimate how much the
Carleson norm of the coefficients changes when we move from
the set $\{(x,t)\in\R^{n}; t>\phi(x)\}$ to $\R_+^{n}$. A
computation gives us that if the original constant was $\|\mu\|_{Carl}$, the
new Carleson norm on $\R_+^{n}$ will depend on $\|\nabla\phi\|_{BMO}$ and on $\|\mu\|_{Carl}$. From this the claim follows, as 
the new norm will be small as long as both $\|\mu\|_{Carl}$ and $\|\nabla\phi\|_{BMO}$ are small
enough. In particular, this applies to domains whose boundaries are given locally
by functions with gradient in VMO.

\medskip

Finally, we replace the gradient Carleson condition (\ref{carl})
by a weaker condition for oscillation of the coefficients
(\ref{carlM}). This entails that the gradient $\nabla u$ will no
longer have a well-defined {\em pointwise} non-tangential maximal function
$N$. Instead an averaged version $\widetilde{N}$ defined by
(\ref{NTMaxVar}) must be used.

\begin{theorem}\label{L2reg2} Under the same assumptions as in
Theorem \ref{L2reg} the $(R)_2$ regularity problem for the
operator $L$ is solvable under a weaker
Carleson condition (\ref{carlM}).
\end{theorem}

\begin{proof} The proof uses same idea as \cite[Corollary
2.3]{DPP}, so we shall skip non-essential details. The
procedure outlined in \cite{DPP} implies that for a matrix $A$
satisfying (\ref{carlM}) with ellipticity constant $\Lambda$ one
can find (by mollifying coefficients of $A$) a new ``perturbed" matrix
$\widetilde{A}$, with same ellipticity constant $\Lambda$, such that
$\widetilde{A}$ satisfies (\ref{carl}) and such that
\begin{equation}
\sup\{\delta(X)^{-1}|(A-\widetilde{A})(Y)|^2;\, Y\in
B(X,\delta(X)/2)\}\label{ePert}
\end{equation}
is the density of a Carleson measure. Moreover, if the Carleson norm for matrix $A$
is small (on regions of size at most $\le r_0$), then so are the Carleson norms
of (\ref{carl}) for $\widetilde{A}$ and (\ref{ePert}). Hence by
Theorem \ref{L2reg} the $(R)_2$ regularity problem is solvable for
the operator $\widetilde{L}u=\mbox{div}(\widetilde{A}\nabla u)$.

The solvability of the regularity problem for perturbed operators
satisfying (\ref{ePert}) has been studied in \cite{KP2}. It
follows by \cite[Theorem 2.1]{KP2} that the $L^p$ regularity
problem for the operator $L$ is solvable for some $p>1$. The $p$
for which the solvability of the regularity problem is assessed is
the $p$ such that the $L^{p'}$, $p'=p/(p-1)$,  Dirichlet problem
for the adjoint operator $L^*$ is solvable. Although the results
in \cite{KP2} are stated for symmetric operators, a careful
study of the proof of \cite[Theorem 2.1]{KP2} reveals that what is
really needed is to replace $L$ by its adjoint when the $L^{p'}$
Dirichlet problem is considered.

However by \cite[Theorem 2.2]{DPP} the $L^2$ Dirichlet problem for
$L^*$ is solvable provided the Carleson norm of (\ref{carlM}) (and
hence (\ref{carl}) for $\widetilde{A}$) is sufficiently small.
Hence we have solvability of the regularity problem $(R)_2$ by
\cite[Theorem 2.1, Remark 2.3]{KP2}.
\end{proof}

\section{The Square Function Revisited}

In this section we revisit bounds for the square function of
$\nabla u$ from the perspective of the Neumann problem. As in
Section 2 we shall assume that $\Omega$ is a smooth domain and we continue to use
the notation we introduced there. Recall that $\Omega_{t_0}$ denotes
the collar neighborhood of the boundary $\partial\Omega\times
(0,t_0)$.

On $\Omega_{t_0}$ we have a well-defined co-normal derivative of
$u$ with respect to the operator $L$; in the metric
$d\sigma\otimes dt$ this is just
$$H=\sum_{i=1}^n a_{ni}\partial_i u,$$
where $(a_{ij})$ are coefficients of the matrix $A$ in local
coordinates near the boundary.\vglue2mm

We have the following key lemma bounding the non-tangential
maximal function of $\nabla u$ by the square function of $H$.

\begin{lemma}\label{NPl1} Let $p\ge 2$. Under the assumptions of Lemma \ref{l1}
there exists $\varepsilon>0$ such that if
$\|\mu\|_{Carl}<\varepsilon$ then for some
$K=K(\Omega,\Lambda,n,p)>0$
\begin{eqnarray} \label{NPe1}
&&\int_{\partial\Omega} N^p(\nabla u)\,dx\le K
\iint_{\Omega_{2r}}|\nabla_T u|^{p-2}|\nabla
H|^2\delta(X)\,dX+C(r)\iint_{\Omega\setminus\Omega_{r/2}}|\nabla
u|^p\,dX.
\end{eqnarray}
\end{lemma}

\begin{proof} We mainly work in the collar neighborhood
$\Omega_{t_0}$ defined above. We choose $r\le t_0/5$. Using the
results we have on the solvability of the regularity problem we
know that for sufficiently small $\varepsilon>0$ we have:
\begin{eqnarray} \label{NPE1}
&&\int_{\partial\Omega} N^p(\nabla u)\,dx\le K
\int_{\partial\Omega} |\nabla_T u|^p\,dx.
\end{eqnarray}

Since $\partial\Omega$ is a smooth compact manifold, there is a
finite collection of balls $Q_1,Q_2,\dots, Q_k$ in $\R^{n-1}$ of
diameter comparable to $r$ and smooth diffeomorphisms
$\varphi_s:5Q_s\to\dom$ such that $\bigcup_s\varphi_s(9/8Q_s)$
covers $\dom$. Here $rQ$ denotes the concentric enlargement of $Q$
by a factor of $r$. Let us also find smooth partition of unity
$\phi_s$ such
$$\sum \phi_s=1\text{ on }\dom,\qquad \phi_s=1\mbox{ on }Q_s\qquad\text{and}\qquad \text{supp }\phi_s\subset{9/8Q_s}.$$

Let us fix $s$ and work on one ball $Q=Q_s$ and
$\xi_1=\phi_s$. We may assume that $|\xi_1'|\le C/r$. Choose
another function $\xi_2:[0,\infty)\to \mathbb R$ such that
$\xi_2(t)=1$ on $[0,r]$, $|\xi_2'|\le 5/r$ and support contained
in $[0,2r]$. Now define
\begin{equation} \label{cutoff}
\xi(X)=\xi(x,t)=\xi_1(x)\xi_2(t).
\end{equation}

We work on estimating righthand side of (\ref{NPE1}) in local
coordinates on $5Q\times (0,5r)$. Denote by $v_k=\partial_ku$ for
$k=1,2,\dots,n$. For each $k\le n-1$ we have
\begin{eqnarray}
\nonumber \int_{{\mathbb
R}^{n-1}}|v_k(x,0)|^p\xi_1(x)dx&=&-\iint_{{\mathbb
R}^{n}_+}\partial_n(|v_k|^p\xi)(X)\,dX\\
\label{NPe2} &=&-p\iint_{{\mathbb
R}^{n}_+}|v_k|^{p-2}v_k(\partial_nv_k)\xi\,dX-\iint_{{\mathbb
R}^{n}_+}|v_k|^p\xi_1\xi_2'\,dX=I+II.
\end{eqnarray}

The second term on the right-hand side of (\ref{NPe2}) is
controlled by $\iint_K|\nabla u|^p$ where $K=\{X=(x,t);x\in
5Q\text{ and }r/2\le t\le 5r\}$. We deal with the first term.
Since $\partial_nv_k=\partial_kv_n$ we have

\begin{eqnarray}
I&=&-p\iint_{{\mathbb R}^{n}_+}|v_k|^{p-2}v_k(\partial_kv_n)\xi\,dX\nonumber\\
&=&-p\iint_{{\mathbb
R}^{n}_+}|v_k|^{p-2}v_k\partial_k\left(\frac{a_{ni}}{a_{nn}}v_i\right)\xi\,dX+p\sum_{i<n}\iint_{{\mathbb
R}^{n}_+}|v_k|^{p-2}v_k\partial_k\left(\frac{a_{ni}}{a_{nn}}v_i\right)\xi\,dX.\label{NPe3}
\end{eqnarray}

The second term of (\ref{NPe3}) can be further written as
\begin{eqnarray}
&&p\sum_{i<n}\iint_{{\mathbb
R}^{n}_+}|v_k|^{p-2}v_k\partial_k\left(\frac{a_{ni}}{a_{nn}}v_i\right)\xi\,dX\nonumber\\
&=&p\sum_{i<n}\iint_{{\mathbb
R}^{n}_+}|v_k|^{p-2}v_kv_i\partial_k\left(\frac{a_{ni}}{a_{nn}}\right)\xi\,dX+\sum_{i<n}\iint_{{\mathbb
R}^{n}_+}\partial_i(|v_k|^{p})\frac{a_{ni}}{a_{nn}}\xi\,dX.\label{NPe4}
\end{eqnarray}
We introduce $(\partial_n t)$ into both the terms of (\ref{NPe4})
and integrate by parts. This gives
\begin{eqnarray}
&&-\sum_{i<n}\left[p\iint_{{\mathbb
R}^{n}_+}\partial_n\left(|v_k|^{p-2}v_kv_i\partial_k\left(\frac{a_{ni}}{a_{nn}}\right)\xi\right)
t\,dX+\iint_{{\mathbb
R}^{n}_+}\partial_n\left(\partial_i(|v_k|^{p})\frac{a_{ni}}{a_{nn}}\xi\right)
t\,dX\right]\nonumber\\
&=&-\sum_{i<n}\left[p\iint_{{\mathbb
R}^{n}_+}\partial_n(|v_k|^{p-2}v_k)v_i\partial_k\left(\frac{a_{ni}}{a_{nn}}\right)\xi
t\,dX+p\iint_{{\mathbb
R}^{n}_+}|v_k|^{p-2}v_k\partial_n(v_i)\partial_k\left(\frac{a_{ni}}{a_{nn}}\right)\xi
t\,dX                 \right. \nonumber\\
&&\hskip10mm +\left. \iint_{{\mathbb
R}^{n}_+}\partial_i(|v_k|^{p})\partial_n\left(\frac{a_{ni}}{a_{nn}}\right)\xi
t\,dX\right]\label{NPe5}\\
&&-\sum_{i<n}\left[p\iint_{{\mathbb
R}^{n}_+}|v_k|^{p-2}v_kv_i\partial_k\left(\frac{a_{ni}}{a_{nn}}\right)\xi_1\xi_2'
t\,dX+\iint_{{\mathbb R}^{n}_+}\partial_i(|v_k|^{p})\frac{a_{ni}}{a_{nn}}\xi_1\xi_2' t\,dX\right]\nonumber\\
&&-\sum_{i<n}\left[p\iint_{{\mathbb
R}^{n}_+}|v_k|^{p-2}v_kv_i\partial_n\partial_k\left(\frac{a_{ni}}{a_{nn}}\right)\xi
t\,dX+\iint_{{\mathbb
R}^{n}_+}\partial_n\partial_i(|v_k|^{p})\frac{a_{ni}}{a_{nn}}\xi
t\,dX\right]\nonumber.
\end{eqnarray}

The last two terms we integrate by parts one more time as we
switch the order of derivatives. This gives
\begin{eqnarray}
&&\sum_{i<n}\left[p\iint_{{\mathbb
R}^{n}_+}\partial_k\left(|v_k|^{p-2}v_kv_i\xi\right)\partial_n\left(\frac{a_{ni}}{a_{nn}}\right)
t\,dX+\iint_{{\mathbb
R}^{n}_+}\partial_i\left(\frac{a_{ni}}{a_{nn}}\xi\right)\partial_n
(|v_k|^p) t\,dX\right]\nonumber.\\
&=&\sum_{i<n}\left[p\iint_{{\mathbb
R}^{n}_+}\partial_k(|v_k|^{p-2}v_k)v_i\partial_n\left(\frac{a_{ni}}{a_{nn}}\right)\xi
t\,dX+p\iint_{{\mathbb R}^{n}_+}|v_k|^{p-2}v_k(\partial_k
v_i)\partial_n\left(\frac{a_{ni}}{a_{nn}}\right)\xi
t\,dX                 \right. \nonumber\\
&&\hskip2mm +\left. \iint_{{\mathbb R}^{n}_+}\partial_n
(|v_k|^p)\partial_i\left(\frac{a_{ni}}{a_{nn}}\right)\xi
t\,dX\right]\label{NPe6}\\
&&+\sum_{i<n}\left[p\iint_{{\mathbb
R}^{n}_+}|v_k|^{p-2}v_kv_i\partial_n\left(\frac{a_{ni}}{a_{nn}}\right)(\partial_k\xi_1)\xi_2
t\,dX+\iint_{{\mathbb R}^{n}_+}\partial_n
(|v_k|^p)\frac{a_{ni}}{a_{nn}}(\partial_i\xi_1)\xi_2
t\,dX\right]\nonumber
\end{eqnarray}

The first three terms on the righthand side of both (\ref{NPe5}) and
(\ref{NPe6}) can be bounded from above by
\begin{eqnarray}
&&C\iint_{2Q\times[0,2r]}|\nabla u|^{p-1}|\nabla^2 u||\nabla
A|t\,dX\label{NPe10}\\
&\le& \left(\iint_{2Q\times[0,2r]}|\nabla u|^{p-2}|\nabla^2
u|^2t\,dX\right)^{1/2}\left(\iint_{2Q\times[0,2r]}|\nabla
u|^{p}|\nabla
A|^2t\,dX\right)^{1/2}\nonumber\\
&\le& \left(\int_{2Q}N(\nabla u)^{p-2}\iint_{\Gamma(x)}|\nabla^2
u(X)|^2{t}^{2-n}\,dX\,dx\right)^{1/2}\left(\iint_{2Q\times[0,2r]}|\nabla
u|^{p}|\nabla A|^2t\,dX\right)^{1/2}\nonumber\\
&\le& \left(\int_{2Q}N(\nabla u)^{p-2}S^2(\nabla
u)\,dx\right)^{1/2}\|\mu\|_{Carl}^{1/2} \|N(\nabla
u)\|^{p/2}_{L^p(2Q)}\nonumber\\ &=&\|\mu\|_{Carl}^{1/2}\|S(\nabla
u)\|_{L^p(2Q)}\|N(\nabla u)\|_{L^p(2Q)}^{p-1}.\nonumber
\end{eqnarray}

The fourth term on righthand side of (\ref{NPe5}) can be estimated
by
\begin{eqnarray}
&&C\iint_{2Q\times[r,2r]}|\nabla u|^{p}|\nabla
A|\textstyle{\frac{t}{r}}\,dX\label{NPe11}\\
&\le& \left(\iint_{2Q\times[r,2r]}|\nabla
u|^{p}{\textstyle\frac{t}{r^2}}\displaystyle\,dX\right)^{1/2}\left(\iint_{2Q\times[0,2r]}|\nabla
u|^{p}|\nabla
A|^2t\,dX\right)^{1/2}\nonumber\\\nonumber&\le&\left(\int_{2Q}
N(\nabla u)^p(x)\,dx \right)^{1/2}  \|\mu\|_{Carl}^{1/2} \|N(\nabla
u)\|^{p/2}_{L^p(2Q)}=\|\mu\|_{Carl}^{1/2} \|N(\nabla u)\|^{p}_{L^p(2Q)}.
\end{eqnarray}

The fifth term on righthand side of (\ref{NPe5}) can be estimated
by
\begin{eqnarray}
&&C\iint_{2Q\times[r,2r]}|\nabla u|^{p-1}|\nabla^2
u|\textstyle{\frac{t}{r}}\,dX\nonumber\\
&\le& \left(\iint_{2Q\times[r,2r]}|\nabla
u|^{p}\displaystyle\,dX\right)^{p/(p-1)}\left(\iint_{2Q\times[0,2r]}|\nabla^2
u|^{p}\,dX\right)^{1/p}\label{NPe12}\\
&\le& C(r)\iint_K|\nabla u|^p\,dX.\nonumber
\end{eqnarray}
To get the last line we used some standard elliptic estimates away
from the boundary (for example, it is sufficient to generalise Caccioppoli's inequality to inhomogeneous equations via the proof in \cite[p.~2]{K}). By the Carleson condition we have $|\nabla
A|\le \|\mu\|_{Carl}^{1/2}/r$ there. The rest is a standard bootstrap argument using the
equation ${\bf v}=\nabla u$ satisfies, i.e., $L{\bf
v}=\div((\nabla A){\bf v})$ eventually yielding $L^p$ bounds on
$\nabla\bf v$ in $K$.

We denote the co-normal derivative of $u$ by
$H=\sum_{i}a_{ni}\partial_i u=\sum_i a_{ni}v_i$ and write the
first term of (\ref{NPe3}) as

\begin{eqnarray} \nonumber &&-p\iint_{{\mathbb
R}^{n}_+}|v_k|^{p-2}v_k\partial_k\left(\frac{H}{a_{nn}}\right)\xi\,dX=-p\iint_{{\mathbb
R}^{n}_+}|v_k|^{p-2}v_k\partial_k\left(\frac{H}{a_{nn}}\right)\xi(\partial_n
t)\,dX\\
&=&\nonumber p\iint_{{\mathbb
R}^{n}_+}\partial_n(|v_k|^{p-2}v_k)\partial_k\left(\frac{H}{a_{nn}}\right)\xi
t\,dX+ p\iint_{{\mathbb
R}^{n}_+}|v_k|^{p-2}v_k\partial_k\left(\frac{H}{a_{nn}}\right)\xi_1\xi_2'
t\,dX\\
&&+\label{NPe7} p\iint_{{\mathbb
R}^{n}_+}|v_k|^{p-2}v_k\partial_n\partial_k\left(\frac{H}{a_{nn}}\right)\xi
t\,dX,
\end{eqnarray}
where the last term further yields:
\begin{eqnarray}
&&\label{NPe8} -p\iint_{{\mathbb
R}^{n}_+}\partial_k(|v_k|^{p-2}v_k)\partial_n\left(\frac{H}{a_{nn}}\right)\xi
t\,dX- p\iint_{{\mathbb
R}^{n}_+}|v_k|^{p-2}v_k\partial_n\left(\frac{H}{a_{nn}}\right)(\partial_k\xi_1)\xi_2
t\,dX.
\end{eqnarray}
If the derivative in the first two terms on the righthand side of
(\ref{NPe7}) and (\ref{NPe8}) falls on the coefficients of the
matrix $A$ we obtain terms we have already bounded above (see
(\ref{NPe10}) and (\ref{NPe11})). If the derivative falls on $H$
the first term on the righthand side of both (\ref{NPe7}) and
(\ref{NPe8}) is bounded by
\begin{eqnarray}
&&\nonumber C\iint_{{\mathbb R}^{n}_+}|v_k|^{p-2}|\nabla
v_k||\nabla H|\xi t\,dX\\
&\le&C\nonumber\left(\iint_{{\mathbb R}^{n}_+}|v_k|^{p-2}|\nabla
v_k|^2\xi t\,dX\right)^{1/2} \left(\iint_{{\mathbb
R}^{n}_+}|v_k|^{p-2}|\nabla H|^2\xi t\,dX\right)^{1/2}\\\nonumber
&\le&C\left(\int_{2Q}N^{p-2}(v_k)(x)\int_{\Gamma(x)}|\nabla
v_k(X)|^2t^{2-n} dX\,dx\right)^{1/2}\left(\iint_{\R^n_+}|\nabla_T
u|^{p-2}|\nabla
H|^2 \xi t\,dX\right)^{1/2}\\
&=&\nonumber C\left(\int_{2Q}N^{p-2}(v_k)(x)
S^2(v_k)(x)\,dx\right)^{1/2}\left(\iint_{\R^n_+}|\nabla_T
u|^{p-2}|\nabla
H|^2\xi t\,dX\right)^{1/2}\\
&=&\nonumber
C\|N(v_k)\|_{L^p(2Q)}^{p/2-1}\|S(v_k)\|_{L^p(2Q)}\left(\iint_{\R^n_+}|\nabla_T
u|^{p-2}|\nabla H|^2 \xi t\,dX\right)^{1/2}.
\end{eqnarray}

If the derivative falls on $H$ in the second term of (\ref{NPe7}),
we get terms of the same form as (\ref{NPe11}) and (\ref{NPe12}).

It follows that for all $k\le n-1$ we have
\begin{eqnarray}
\label{NPe13} &&\int_{{\mathbb R}^{n-1}}v_k(x,0)^p\xi_1(x)dx\\
\nonumber &\le&\nonumber
C\|N(v_k)\|_{L^p(2Q)}^{p/2-1}\|S(v_k)\|_{L^p(2Q)}\left(\iint_{\R^n_+}|\nabla_T
u|^{p-2}|\nabla H|^2\xi t\,dX\right)^{1/2}\\
&&\nonumber+\|\mu\|_{Carl}^{1/2}\|N(\nabla
u)\|_{L^p(2Q)}^{p-1}\left[\|S(\nabla u)\|_{L^p(2Q)}+\|N(\nabla u)\|_{L^p(2Q)}\right]\\
&&\nonumber+
C(r)\iint_K|\nabla u|^p\,dX+E.
\end{eqnarray}
Here $E$ denotes remainder terms; these are the last two terms of
(\ref{NPe6}) and the last term of (\ref{NPe8}) when the derivative
falls on $H$. We now sum (\ref{NPe13}) over all $k\le n-1$ and
also sum over all coordinate patches $Q_s$. We notice that the
error terms $E$ with complete cancel out as $\sum_s (\partial_k
\phi_s)=0$ where $(\phi_s)$ is the partition of unity we
considered above. This yields a global estimate
\begin{eqnarray}
\nonumber &&\int_{\partial\Omega}|\nabla_T u|^pdx\le
C\|N(v_k)\|_{L^p(\partial\Omega)}^{p/2-1}\|S(v_k)\|_{L^p(\partial\Omega)}\left(\iint_{\Omega_{2r}}|\nabla_T
u|^{p-2}|\nabla H|^2\delta(X)\,dX\right)^{1/2}\\\nonumber&&+
\|\mu\|_{Carl}\|N(\nabla
u)\|_{L^p(\partial\Omega)}^{p-1}\left[\|S(\nabla u)\|_{L^p(\partial\Omega)}+\|N(\nabla
u)\|_{L^p(\partial\Omega)}\right]+
C(r)\iint_{\Omega\setminus\Omega_{r/2}}|\nabla u|^p\,dX.
\end{eqnarray}

From this, by (\ref{NPE1}) and using Lemma \ref{lSNp}, we get that
for all sufficiently small $\|\mu\|_{Carl}<\varepsilon$ the desired estimate
(\ref{NPe1}) holds.
\end{proof}

\begin{lemma}\label{NPl2} Let $p\ge 2$ be an integer, $k$ be an integer such that $0\le k\le p-2$.
Under the assumptions of Lemma \ref{l1} there exists $\varepsilon>0$ such that if
$\|\mu\|_{Carl,r_0}<\varepsilon$ then for some constant
$K=K(\Omega,\Lambda,n,k)>0$
\begin{eqnarray} \label{NPe14}
&&\iint_{\Omega_{r}}|\nabla_T u|^{p-k-2}|H|^{k}|\nabla
H|^2\delta(X)\,dX\\&\le&\nonumber
K(p-k-2)\iint_{\Omega_{2r}}|\nabla_T u|^{p-k-3}|H|^{k+1}|\nabla
H|^2\delta(X)\,dX+C(r)\iint_{\Omega\setminus\Omega_{r}}|\nabla
u|^{p}\,dX\\\nonumber&&+K\int_{\partial\Omega}|H|^{p}\,dx.
\end{eqnarray}
\end{lemma}

\begin{proof} We will establish (\ref{NPe14}) by induction on $k$.
If $k=0$ by Lemma \ref{NPl1} we have:
\begin{eqnarray} \label{NPe15}
&&\int_{\partial\Omega} N^{p}(\nabla u)\,dx\le K
\iint_{\Omega_{r}}|\nabla_T u|^{p-2}|\nabla
H|^2\delta(X)\,dX+C(r)\iint_{\Omega\setminus\Omega_{r/2}}|\nabla
u|^p\,dX.
\end{eqnarray}

For $k>0$ we use (\ref{NPe1}) and the induction assumption
(\ref{NPe14}) for all indices $0,1,\dots,k-1$. This gives

\begin{eqnarray} \nonumber
\int_{\partial\Omega} N^{p}(\nabla u)\,dx&\le& K
\iint_{\Omega_{r}}|\nabla_T u|^{p-k-2}|H|^{k}|\nabla
H|^2\delta(X)\,dX+C(r)\iint_{\Omega\setminus\Omega_{r/2}}|\nabla
u|^p\,dX\\\label{NPe16}&&+K\int_{\partial\Omega}|H|^{p}\,d\sigma.
\end{eqnarray}
Here $K=K(k)$ and (\ref{NPe16}) holds for all sufficiently small
$\varepsilon>0$. From this, the inequality

\begin{eqnarray} &&\nonumber
\int_{\partial\Omega} N^{p}(\nabla u)\,dx+K
\iint_{\Omega_{r}}|\nabla_T u|^{p-k-2}|H|^{k}|\nabla
H|^2\delta(X)\,dX\\\label{NPe16a} &\le& 2K
\iint_{\Omega_{r}}|\nabla_T u|^{p-k-2}|H|^{k}|\nabla
H|^2\delta(X)\,dX\\&&+2C(r)\iint_{\Omega\setminus\Omega_{r/2}}|\nabla
u|^p\,dX\nonumber+2K\int_{\partial\Omega}|H|^{p}\,d\sigma
\end{eqnarray}
holds when $k=0$ without any further assumptions, and when $k > 0$ under the induction hypotheses (\ref{NPe14}) for indices $0,1,\dots,k-1$.
Let us choose a cutoff function $\xi$ as in (\ref{cutoff}). To control
\begin{eqnarray} \nonumber
\iint_{\Omega_{r}}|\nabla_T u|^{p-k-2}|H|^{k}|\nabla
H|^2\delta(X)\,dX
\end{eqnarray}
it suffices to control
\begin{eqnarray} \nonumber
\iint_{\R^n_+}|\nabla_T
u|^{p-k-2}|H|^{k}b_{ij}(\partial_iH)(\partial_jH)\xi t\,dX=I
\end{eqnarray}
for some matrix $B$ satisfying the
ellipticity condition to be specified later.

We integrate this by parts. This gives
\begin{eqnarray} \nonumber
I&=&-\frac1{k+1} \iint_{\R^n_+}|\nabla_T
u|^{p-k-2}|H|^{k}H\partial_i (b_{ij}\partial_jH)\xi t\,dX\\
&&\nonumber -\frac1{k+1} \iint_{\R^n_+}|\nabla_T
u|^{p-k-2}|H|^{k}Hb_{nj}(\partial_jH)\xi\,dX\\
&&\label{NPe17} -\frac1{k+1} \iint_{\R^n_+}|\nabla_T
u|^{p-k-2}|H|^{k}Hb_{nj}(\partial_jH)(\partial_i\xi) t\,dX\\
&&\nonumber -\frac{p-k-2}{k+1} \iint_{\R^n_+}|\nabla_T
u|^{p-k-4}(\nabla_Tu\cdot
\partial_i(\nabla_T u))|H|^{k}Hb_{nj}(\partial_jH)\xi t\,dX .
\end{eqnarray}
The second term only appears in (\ref{NPe17}) if $i=n$ as
the function $t$ obviously only depends on the variable $x_n=t$.
We first deal with the third term of (\ref{NPe17}) when $i=n$. As $|\xi_2'|\le
2/r$ and $\xi_2'=0$ on $[0,r]$ we have that this term is bounded by
\begin{eqnarray}
\label{NPe17a} \iint_{Q \times [r,2r]}|\nabla_T
u|^{p-k-2}|H|^{k+1}|\nabla
H|\textstyle\frac{t}{r}\displaystyle\,dX &\le& \varepsilon^{1/2}\int_{2Q}
N_{2r}^{p}(\nabla u)\,dx +\\
\nonumber &&C(r)\iint_{\Omega\setminus\Omega_r}|\nabla u|^p\,dX,
\end{eqnarray}
since this term is of same type as (\ref{NPe11}) and (\ref{NPe12})
it can be estimated as before. Now for the terms with $i < n$ in
the third term of (\ref{NPe17}) we observe that they will cancel
when we sum over the index $s$ in the partition of unity
introduced via the cutoff function $\xi$ from (\ref{cutoff}). We
choose the matrix $B$ so that $b_{nn}=1$. Then the second term of
(\ref{NPe17}) if $j=n$ looks like
\begin{eqnarray} &&\nonumber -\frac{1}{(k+1)(k+2)} \iint_{\R^n_+}|\nabla_T
u|^{p-k-2}(\partial_n|H|^{k+2})\xi\,dX\\
&=&\nonumber -\frac{1}{(k+1)(k+2)}
\iint_{\R^n_+}\partial_n(|\nabla_T
u|^{p-k-2}|H|^{k+2}\xi)\,dX\\
&&+\label{NPe18}\frac{1}{(k+1)(k+2)}
\iint_{\R^n_+}\partial_n(|\nabla_T
u|^{p-k-2})|H|^{k+2}\xi\,dX\\
&&\nonumber +\frac{1}{(k+1)(k+2)} \iint_{\R^n_+}|\nabla_T
u|^{p-k-2}|H|^{k+2}\xi'\,dX.
\end{eqnarray}
Here the last term again can be estimated by a solid integral
$C(r)\iint_{\Omega\setminus\Omega_r}|\nabla u|^p\,dX$ in the
interior of the domain. The first term is equal to a boundary
integral
\begin{eqnarray} \nonumber \frac{1}{(k+1)(k+2)}
\int_{\partial\Omega}|\nabla_T u|^{p-k-2}|H|^{k+2}\,dX\le \eta
\|\nabla_T
u\|^p_{L^p(\partial\Omega)}+C(\eta)\|H\|^p_{L^p(\partial\Omega)},
\end{eqnarray}
for $\eta>0$ arbitrary small. Note that
$$\eta
\|\nabla_T u\|^p_{L^p(\partial\Omega)}\le\eta \|N(\nabla
u)\|^p_{L^p(\partial\Omega)}.$$ We choose $\eta>0$ so small that
we can hide the term $\eta \|N(\nabla
u)\|^p_{L^p(\partial\Omega)}$ on lefthand side of
(\ref{NPe16a}).

It remains to deal with the second term of (\ref{NPe18}). We
differentiate and change the order of derivatives $\partial_n$ and
$\nabla_T$:
\begin{eqnarray} &&\label{NPe19}\frac{p-k-2}{(k+1)(k+2)}
\iint_{\Omega_{2r}}|\nabla_T u|^{p-k-4}(\nabla_Tu\cdot
\nabla_T\partial_n u)|H|^{k+2}\xi \,dX.
\end{eqnarray}
We reintroduce the co-normal derivative $H$ as $\partial_n
u=\frac{H}{a_{nn}}-\sum_{j<n}\frac{a_{nj}}{a_{nn}}v_j$. We also
insert a term $(\partial_n t)=1$ into both integrals. Then we
integrate by parts again in the $\partial_n$ derivative. Whenever
exactly one derivative falls on the coefficients (either $a_{nn}$
or $\frac{a_{nj}}{a_{nn}}$) those terms are bounded by
\begin{equation}\label{NPe20}
\iint_{2Q\times[0,2r]}|\nabla A||\nabla u|^{p-1}|\nabla^2 u|t\,dX
\end{equation}
which is the term of type (\ref{NPe10}) and has therefore a bound
of type\newline $\varepsilon^{1/2}\|S(\nabla u)\|_{L^p(2Q)}\|N(\nabla
u)\|_{L^p(2Q)}^{p-1}$, with $\varepsilon$ bounding the
Carleson norm of the coefficients. For sufficiently small $\varepsilon$, thanks to
Lemma \ref{lSNp}, this can be hidden on the lefthand side of
(\ref{NPe16a}).

If both $\partial_n$ and $\nabla_T$ derivative fall on
coefficients, there are two possibilities. The first possibility
is that they fall on the same coefficient and so then we do a
further integration by parts in $\nabla_T$ moving this derivative
on other terms. This again will yield term of type (\ref{NPe20}).
The second possibility is that they fall on separate coefficients and so take
the form (\ref{form1}), which can be estimated appropriately with
the help of Lemma \ref{lSNp}. We obtain another error term when
$\partial_n$ falls on $\xi$, however in that case we get a term of
type (\ref{NPe17a}) we handled before. Let us deal with the term
when both derivatives fall on $H$. In that case we have
\begin{eqnarray} &&\label{NPe21}-\frac{p-k-2}{(k+1)(k+2)}
\iint_{\Omega_{2r}}\frac1{a_{nn}}|\nabla_T
u|^{p-k-4}(\nabla_Tu\cdot \nabla_T\partial_n H)|H|^{k+2}\xi t\,dX.
\end{eqnarray}
We move the $\nabla_T$ derivative off $\partial_n H$. We can get a
term of type (\ref{NPe20}) and two terms that can be dominated by
\begin{eqnarray} &&\label{NPe22}C(p-k-2)
\iint_{\Omega_{2r}}|\nabla_T u|^{p-k-4}|\nabla(\nabla_Tu)||\nabla
H||H|^{k+2} t\,dX\\
&+&C(p-k-2)\nonumber\iint_{\Omega_{2r}}|\nabla_T u|^{p-k-3}|\nabla
H|^2|H|^{k+1} t\,dX.
\end{eqnarray}
Also, when $\nabla_T$ lands on $\xi$ we get error terms which will cancel when we sum over coordinate patches. Observe also that the last term of (\ref{NPe17}) can be controlled by
\begin{equation} \label{nov1}
C(p-k-2) \iint_{\Omega_{2r}}|\nabla_T u|^{p-k-3}|\nabla(\nabla_Tu)||\nabla
H||H|^{k+1} t\,dX
\end{equation}
We now deal with the terms arising from
$-\sum_{j<n}\frac{a_{nj}}{a_{nn}}v_j$. Here we write
\[
\nabla_T\left(\sum_{j<n}\frac{a_{nj}}{a_{nn}}v_j\right) = \sum_{j<n}\nabla_T\left(\frac{a_{nj}}{a_{nn}}\right)\partial_ju + \sum_{j<n}\frac{a_{nj}}{a_{nn}}\partial_j(\nabla_Tu).
\]
The contribution of the first term here, when substituted in \eqref{NPe19}, can be dealt with by again introducing the factor $\partial_n t$ and integrating by parts. When $\partial_n$ lands on $\nabla_T(a_{nj}/a_{nn})$, we can move the tangential derivates off by again integrating by parts. All this yields terms of the form \eqref{form1} (with exponent $p$ instead of $2$) and \eqref{NPe20}, which can be controlled appropriately. Substituting the second term in \eqref{NPe19} yields
\begin{eqnarray} &&\label{NPe23}\frac{1}{(k+1)(k+2)}\sum_{j<n}
\iint_{\Omega_{2r}}\frac{a_{nj}}{a_{nn}}\partial_j(|\nabla_T
u|^{p-k-2})|H|^{k+2}\xi (\partial_n t)\,dX.
\end{eqnarray}
Moving $\partial_n$ across using integration by parts and if necessary moving $\partial_j$
we obtain terms either bounded by \eqref{form1} (with exponent $p$ instead of $2$), \eqref{NPe20}, \eqref{NPe22} or \eqref{nov1}. Thus the analysis of the
second term of (\ref{NPe17}) for $j=n$ reduces to controlling \eqref{NPe22} and \eqref{nov1}, a task which we will postpone for now. When $j<n$ in the second
term of (\ref{NPe17}) we again introduce $(\partial_n t)$. This
gives
\begin{eqnarray}
&&\nonumber -\iint_{\Omega_{2r}}|\nabla_T
u|^{p-k-2}b_{nj}\partial_j(|H|^{k+2})\xi (\partial_n t)\,dX
\end{eqnarray}
We integrate by parts. When $\partial_n$ falls on $|\nabla_T
u|^{p-k-2}$ we can dominate such a term by \eqref{nov1},
when $\partial_n$ falls on $b_{nj}$ we obtain a terms of type \eqref{form1} (with exponent $p$ instead of $2$) and
\eqref{NPe20} and, provided we choose matrix $B$ so that coefficients
of $B$ also satisfy the Carleson condition. If $\partial_n$ hits
$\xi$ we get terms which can be bounded by \eqref{NPe11} and \eqref{NPe12}. Finally the remaining term
is
\begin{eqnarray}
&&\nonumber \iint_{\Omega_{2r}}|\nabla_T
u|^{p-k-2}b_{nj}\partial_j\partial_n(|H|^{k+2})\xi t\,dX.
\end{eqnarray}
We integrate by parts again in $\partial_j$ giving us terms
of type \eqref{form1} (with exponent $p$ instead of $2$), \eqref{NPe20}, \eqref{NPe22} and \eqref{nov1}. The only remaining terms
we have not yet bounded are the first term of \eqref{NPe17},
\eqref{NPe22} and \eqref{nov1}. The second term of (\ref{NPe22}) is already of
desired form (see righthand side of (\ref{NPe14})). By the
Cauchy-Schwarz inequality, the first term of (\ref{NPe22}) can be bounded by
\begin{eqnarray} \nonumber C(p-k-2)&&\left(
\iint_{\Omega_{2r}}|\nabla_T u|^{p-k-3}|\nabla
H|^2|H|^{k+1} t\,dX\right)^{1/2}\times\\
&&\left( \iint_{\Omega_{2r}}|\nabla_T
u|^{p-k-5}|\nabla(\nabla_Tu)|^2|H|^{k+3}
t\,dX\right)^{1/2}\nonumber\\
\le C(p-k-2)&&\left( \iint_{\Omega_{2r}}|\nabla_T
u|^{p-k-3}|\nabla H|^2|H|^{k+1} t\,dX\right)^{1/2}\|N(\nabla
u)\|^{p/2-1}_{L^p(\partial\Omega)}\|S(\nabla
u)\|_{L^p(\partial\Omega)}.\nonumber
\end{eqnarray}
The last line can be further bounded by
$$\eta\|N(\nabla
u)\|^{p}_{L^p(\partial\Omega)} +
C(\eta)(p-k-2)^2\iint_{\Omega_{2r}}|\nabla_T u|^{p-k-3}|\nabla
H|^2|H|^{k+1} t\,dX,$$ for $\eta>0$ arbitrary small. Hence as
before we can hide $\eta\|N(\nabla u)\|^{p}_{L^p(\partial\Omega)}$
on the lefthand side of (\ref{NPe16a}). Term \eqref{nov1} can be dealt with in a very similar fashion. We summarize what we have
so far. By (\ref{NPe16a}) and all estimates above we have

\begin{eqnarray} &&\nonumber
\alpha\int_{\partial\Omega} N^{p}(\nabla u)\,d\sigma+
\iint_{\Omega_{r}}|\nabla_T u|^{p-k-2}|H|^{k}|\nabla
H|^2\delta(X)\,dX\\\label{NPe24} &\le& K(p-k-2)
\iint_{\Omega_{r}}|\nabla_T u|^{p-k-3}|H|^{k+1}|\nabla
H|^2\delta(X)\,dX\\&&+2C(r)\iint_{\Omega\setminus\Omega_{r/2}}|\nabla
u|^p\,dX\nonumber+K\int_{\partial\Omega}|H|^{p}\,d\sigma\\
&&-\frac{K}{k+1} \iint_{\Omega_{2r}}|\nabla_T
u|^{p-k-2}|H|^{k}H(\widetilde{L}H)\xi t\,dX.\nonumber
\end{eqnarray}
Here $\widetilde{L}H=\mbox{ div}(B\nabla H)$ and $\alpha>0$. The
precise value of $\alpha$ depends on choice of $\eta>0$ above and
$\varepsilon>0$. Clearly, (\ref{NPe24}) is the desired estimate
(\ref{NPe14}) modulo the last extra term we shall consider
now.\vglue2mm

As above we use the summation convention, we only write the sum explicitly
whenever we do not sum over all indices. For $\widetilde{L}H$ we
have
\begin{eqnarray} &&\nonumber \widetilde{L}H=\partial_i(b_{ij}\partial_j
H)=\sum_{j<n} \partial_i(b_{ij}\partial_j(a_{nk} \partial_k u))+
\partial_i(b_{in}\partial_n(a_{nk} \partial_k u)).
\end{eqnarray}
Since $Lu=0$ we have that $\partial_n(a_{nk} \partial_k
u)=-\sum_{j<n}\partial_j(a_{jk} \partial_k u)$. Hence
\begin{eqnarray} &&\nonumber \widetilde{L}H=\partial_i(b_{ij}\partial_j
H)=\sum_{j<n} [\partial_i(b_{ij}\partial_j(a_{nk} \partial_k u))-
\partial_i(b_{in} \partial_j(a_{jk} \partial_k u))].
\end{eqnarray}
We also swap the role of $i$ and $k$ in the second term. From this
\begin{eqnarray} &&\label{NPe25}\widetilde{L}H=\partial_i(b_{ij}\partial_j
H)=\sum_{j<n} [\partial_i(b_{ij}\partial_j(a_{nk} \partial_k u))-
\partial_k(b_{kn} \partial_j(a_{ji} \partial_i u))].
\end{eqnarray}
We choose $b_{ij}=a_{ji}/a_{nn}$. Notice that this guarantees that
$b_{nn}=1$ and that terms in (\ref{NPe25}) where three derivatives
fall on $u$ vanish as these are the terms:
\begin{eqnarray} &&\label{NPe26}\sum_{j<n} [b_{ij}a_{nk} (\partial_i\partial_j\partial_k u)-
b_{kn}a_{ji} (\partial_i\partial_j\partial_k
u)]=\sum_{j<n}a_{nn}^{-1}(a_{ji}a_{nk}-a_{nk}a_{ji})\partial_i\partial_j\partial_k
u=0.
\end{eqnarray}

We now place (\ref{NPe25}) into last term of (\ref{NPe24}). Given
(\ref{NPe26}) some of the remaining terms are

\begin{eqnarray} &&\label{NPe27}
\iint_{\Omega_{2r}}\sum_{j<n}[b_{ij}(\partial_i\partial_j
a_{ij})(\partial_k u)-b_{kn}(\partial_k\partial_j
a_{ji})(\partial_i u)]|\nabla_T u|^{p-k-2}|H|^{k}H\xi t\,dX
\end{eqnarray}
and the rest can be bounded by
\begin{eqnarray} &&\label{NPe28}
\iint_{\Omega_{2r}}|\nabla u|^{p-1}[|\nabla u||\nabla A||\nabla
B|+|\nabla^2 u||\nabla A||B|+|\nabla^2 u||\nabla B||A|]t\,dX.
\end{eqnarray}
The terms in (\ref{NPe27}) have two derivatives on coefficients $a_{ij}$
however one is $\partial_j$ and $j<n$. We therefore integrate by
parts in $\partial_j$. This yields additional terms, but all are
of the form (\ref{NPe28}). However, by an estimate similar to
(\ref{NPe20}) we get that all the terms of (\ref{NPe28}) are
smaller than $C(\varepsilon)\int_{\partial\Omega} N^p_{3r}(\nabla
u)\,d\sigma$, with $\varepsilon$ being the upper bound of the Carleson norm of the coefficients. Hence for sufficiently small
$\varepsilon$ this term can be hidden in (\ref{NPe24}) within the
term $\alpha\int_{\partial\Omega} N^{p}(\nabla u)\,dx$. This
yields the desired estimate (\ref{NPe14}).
\end{proof}

\section{The $(N)_p$ Neumann Problem}

\begin{theorem}\label{NPL2neu} Let $p\ge 2$ be an integer. Under the assumptions of Theorem \ref{LPreg}
with $L$ satisfying the stronger Carleson condition (\ref{carl}) with norm $\|\mu\|_{Carl,r_0}$ there exists
$\varepsilon=\varepsilon(\Lambda,n,p)>0$ such that if
$\max\{\ell,\|\mu\|_{Carl,r_0}\}<\varepsilon$ then the Neumann problem
\begin{eqnarray}
\nonumber &&Lu=0,\qquad\hskip9mm\mbox{in }\Omega,\\
\nonumber &&A\nabla u\cdot \nu=f,\qquad\mbox{on }\partial\Omega,\\
\nonumber &&N(\nabla u)\in L^p(\partial\Omega),
\end{eqnarray}
is solvable for all $f$ in $L^p(\partial\Omega)<\infty$  with
$\int_{\dom} fd\sigma=0$. Moreover, there exists a constant
$C=C(\Lambda,n,a,p)>0$ such that
\begin{equation}
\|N(\nabla u)\|_{L^p(\partial\Omega)}\le
C\|f\|_{L^p(\partial\Omega)}.\label{NPmest}
\end{equation}
\end{theorem}

\begin{proof}
For any $f$ in the Besov space $B^{2,2}_{-1/2}(\partial\Omega)$
such that $\int_{\dom} fd\sigma=0$ the exists a unique (up to a
constant) $H_1^{2}(\Omega)$ solution by the Lax-Milgram theorem.
Observe that our $f\in L^p(\partial\Omega)\subset
B^{2,2}_{-1/2}(\partial\Omega)$ ($p\ge 2$) so it only remains to
establish the estimate (\ref{NPmest}).

Consider $\varepsilon>0$ and take $\|\mu\|_{Carl,r_0}<\varepsilon$. To keep matters simple let us first consider the case when
$\partial\Omega$ is smooth. In this case Lemmas \ref{NPl1} and
\ref{NPl2} apply directly. If follows that for all small $r$ and
$\varepsilon>0$
\begin{eqnarray} \label{NPe44}
&&\int_{\partial\Omega}N^p(\nabla u)\,d\sigma\le K\int_{\dom}
|A\nabla u\cdot\nu|^p\,d\sigma+C(r)\|\nabla
u\|^p_{L^p(\Omega\setminus\Omega_{r})}.
\end{eqnarray}

Here we are using Lemma \ref{NPl2} for $k=0,1,2,\dots,p-2$ while
observing that for the integer $k=p-2$, the first term on the
righthand side of (\ref{NPe14}) is zero. As $A\nabla u\cdot\nu=f$
we have for non-tangential maximal function
\begin{eqnarray} \label{NPe45}
&&\int_{\partial\Omega}N^p(\nabla u)\,d\sigma\le K\int_{\dom}
|f|^p\,d\sigma+C(r)\|\nabla
u\|^p_{L^p(\Omega\setminus\Omega_{r})}.
\end{eqnarray}

We also observe that we have a pointwise estimates on $\nabla
u(X)$ for all $X$ away from the boundary. There, by the Carleson
condition, we have $|\nabla A|\le \|\mu\|_{Carl,r_0}^{1/2}/r$. The rest is a standard
bootstrap argument using the equation ${\bf v}=\nabla u$
satisfies, i.e., $L{\bf v}=\div((\nabla A){\bf v})$ eventually
yielding pointwise bound on $|\nabla u|$ for
$\{X\in\partial\Omega;\,\mbox{dist}(X,\partial\Omega)\ge r\}$.

This yields
\begin{eqnarray} \label{NPe47}
\|\nabla u\|^p_{L^p(\Omega\setminus\Omega_{r})}\le
C(p)\|u\|^p_{H^{2}_1(\Omega)}\le
C(p)\|f\|^p_{B^{2,2}_{-1/2}(\partial\Omega)}.
\end{eqnarray}

Finally, combining (\ref{NPe45}) and (\ref{NPe47}) we obtain the
desired estimate (\ref{NPmest}).\vglue2mm

Now we turn to the more general case, when $\Omega$ has a Lipschitz
boundary with sufficiently small Lipschitz constant $\ell$. This case
also includes the $C^1$ boundary as in this case $\ell$ can be taken
arbitrary small.

The argument here is the same as the one given in the proof of Theorem \ref{L2reg}.
\end{proof}

\vspace{.2cm}

\begin{tabular}{lll}
Martin Dindo\v{s} & Jill Pipher & David Rule \\
School of Mathematics & Brown University & Link\"opings universitet \\
Edinburgh University & Mathematics Department & Matematiska institutionen \\
Mayfield Road & Box 1917 & 581 83 Link\"oping\\
Edinburgh, EH9 3JZ, UK & Providence, RI 02912, USA & Sweden  \\
{\small \tt m.dindos@ed.ac.uk} & {\small \tt jill pipher@brown.edu} & {\small \tt david.rule@liu.se}

\hfill

\end{tabular}
\end{document}